%% file: abs.tex
\documentclass[12pt]{amsart}

\usepackage{graphicx}
\usepackage{amsmath}
\usepackage{amsthm}
\usepackage{amscd}
\usepackage{amssymb}
\usepackage{verbatim}
\usepackage[mathscr]{eucal}
\usepackage[all]{xy}
\usepackage{setspace}
\usepackage{psfrag}
\usepackage{url}
\usepackage{pinlabel}

\usepackage[colorlinks=true,linkcolor=blue,citecolor=blue,urlcolor=blue]{hyperref}

\title[An absolute $\Z/2$ grading on bordered Heegaard Floer homology]{An absolute $\Z/2$ grading on bordered Heegaard Floer homology}
\author{Ina Petkova}
\address {Department of Mathematics, Dartmouth College\\ Hanover, NH 03755}
\email {ina.petkova@dartmouth.edu}

\theoremstyle{plain}
\newtheorem{theorem}{Theorem}
\newtheorem{proposition}[theorem]{Proposition}
\newtheorem{lemma}[theorem]{Lemma}
\newtheorem{corollary}[theorem]{Corollary}
\newtheorem{claim}[theorem]{Claim}

\newtheorem{question}[theorem]{Question}

\newtheorem{defn}[theorem]{Definition}

\makeatletter
\renewenvironment{proof}[1][\proofname]{\par
  \pushQED{\qed}%
  \normalfont \topsep6\p@\@plus6\p@\relax
  \trivlist
  \item[\hskip\labelsep
        \bfseries
    #1\@addpunct{.}]\ignorespaces
}{%
  \popQED\endtrivlist\@endpefalse
}
\makeatother

\def\remark{{\bf {\bigskip}{\noindent}Remark. }}

\def\ackn{{\bf {\bigskip}{\noindent}Acknowledgments. }}

\def\title{\em}

\def\bar{\overline}

 \usepackage[pdftex,lmargin=1.25in,rmargin=1.25in,tmargin=1in,bmargin=1in]{geometry}

\input{macros}

\begin{document}

\maketitle

\begin{abstract}
Bordered Heegaard Floer homology is an invariant for $3$-manifolds, which associates to a surface $F$ an algebra  $\az$, and to a $3$-manifold $Y$ with boundary, together with an orientation-preserving diffeomorphism $\phi: F\to \bdy Y$, a module over $\az$.
In \cite{dec} we defined relative $\Z/2$ differential gradings on the algebra $\az$ and the modules over it. In this paper, we turn the relative grading into an absolute one,  and show that the resulting $\Z/2$-graded module is an invariant of the bordered 3--manifold. 
\end{abstract}

\section{Introduction}

Heegaard Floer homology is an invariant for closed, oriented $3$-manifolds, defined by Ozsv\'ath and Szab\'o \cite{osz14}. The simplest version  takes the form of a chain complex $\cfhat$ over the integers, which splits into a direct sum by the $\mathrm{spin}^c$ structures of the $3$-manifold. Bordered Heegaard Floer homology is an extension of Heegaard Floer homology to manifolds with boundary \cite{bfh2}, which has  provided powerful gluing techniques for computing the original Heegaard
Floer invariants of closed manifolds and knots. While the Floer invariants for closed manifolds enjoy a nice absolute differential grading, for example by $\Z/2$ or by $\Q$ \cite{osz14, osz6}, there is no similar grading for bordered Heegaard Floer homology.

The idea of the bordered Floer construction is as follows. To a parametrized surface one associates a differential algebra $\az$, where $\zz$ is a way to represent the surface, and to a manifold with parametrized boundary represented by $\zz$ a left type $D$ structure $\cfdhat$ over $\az$, or a right $\ainf$-module $\cfahat$ over $\az$.  Both structures are invariants of the manifold up to homotopy equivalence, and their  tensor product is an invariant of the closed manifold obtained by gluing two bordered manifolds along their boundary, and recovers $\cfhat$.

The bordered Heegaard Floer modules above also split according to the $\mathrm{spin}^c$ structures of the $3$-manifold. The fathers of the bordered theory define a grading on the modules by sets with an action by a non-commutative group, one such set for each $\mathrm{spin}^c$ structure. 
 It is natural to desire a group grading which is defined simultaneously for all $\mathrm{spin}^c$ structures. In this direction, Gripp and Huang recently provided a nice construction of an absolute grading by the set of homotopy classes of non-vanishing vector fields on the bordered $3$-manifold \cite{gh}. The goal of this paper is to introduce an absolute $\Z/2$ grading which is easily computable from a Heegaard diagram.

In \cite{dec} we defined a  $\Z/2$ grading on $\az$, as well as a relative $\Z/2$ grading on the modules over $\az$ which agrees with the relative  Maslov grading mod $2$ after gluing.  The grading comes from an ordering and orientation of the $\alpha$- and $\beta$-curves on a Heegaard diagram. There is more than one choice of how to do that, and a priori one only gets a relative grading. In this paper, we introduce a canonical choice and obtain an absolute $\Z/2$ grading. 

\begin{theorem}\label{grthm}
Given a bordered Heegaard diagram $\HH$, there is an absolute $\Z/2$ grading on $\cfdhat(\HH)$. More precisely, if $\mathfrak S(\HH)$ is the set of generators of $\cfdhat(\HH)$ coming from the Heegaard diagram, then  there is a function $m: \mathfrak S(\HH)\to \Z/2$ such that if $\xx\in \mathfrak S(\HH)$ and $\aaa = a(\brho)\in \cala(-\bdy \HH)$, then $\xx$ and $\aaa$ are homogeneous with respect to the grading, and 
\begin{enumerate}
\item $m(\aaa \xx) =  m(\aaa) + m(\xx)$, and 
\item  $m(\bdy \xx) = m(\xx)-1$.
\end{enumerate} 
\end{theorem}
The grading $m$  on $\cala(-\bdy \HH)$ above is the one we construct in \cite{dec}.
The resulting graded module is an invariant of the bordered 3--manifold. 

\begin{theorem}\label{invariance}
Let  $\HH$ be a bordered Heegaard diagram for a bordered $3$-manifold $Y$. Up to graded homotopy equivalence, the $\Z/2$-graded differential module $\cfdhat(\HH)$ is independent of the choice of sufficiently generic admissible almost complex structure, and provincially admissible Heegaard diagram for $Y$.
\end{theorem}

Recall that the Euler characteristic of $\cfdhat(Y)$ spans (over $\Z$) the vector space $\mathrm{Span}[\cfdhat(Y)] = |H_1(Y, \bdy Y)| \Lambda^k \ker(i_{\ast}:H_1(F(\zz))\to H_1(Y))$  \cite[Theorem 4]{dec}.
 We remark that Theorem \ref{invariance} eliminates the sign indeterminacy, implying that not only $\mathrm{Span}[\cfdhat(Y)]$, but $[\cfdhat(Y)] $ itself is a topological invariant of $Y$.

\begin{corollary}
Let $\HH$ be a provincially admissible bordered Heegaard diagram for a bordered 3-manifold $(Y, \zz, \phi)$. Then $[\cfdhat(\HH)]$ is an invariant of the bordered manifold.
\end{corollary}

One may use a similar approach to define an absolute grading on $\cfahat(\HH)$; we do not do this here, but simply work out in detail the case of $\cfdhat(\HH)$. One may hope that 
after gluing our grading would recover the absolute $\Z/2$ grading on $\cfhat$ defined in \cite{osz14}. However, it was observed by Hanselman that this is not the case; see \cite[Remark 2]{jh-splicing}.

\remark The absolute grading defined in this paper easily generalizes to bimodules, and extends \cite[Remark 1.2]{bimod} to decategorification with $\Z$ coefficients. 
\\

The outline of this  paper is as follows. Section \ref{backgr} provides a brief introduction to bordered Floer homology. Section \ref{closedgr} discusses a grading on the Heegaard Floer homology of closed manifolds, following \cite{dsfh}. Section \ref{rel} extends the definitions and results on the  relative grading on bordered Floer homology from \cite{dec}. Section \ref{abssec} resolves the indeterminacy in the definition of the grading, lifting it to an absolute grading. Section \ref{inv} is the proof of Theorem \ref{invariance}.

\ackn I am grateful to Robert Lipshitz for many inspiring conversations, and for his valuable comments on earlier drafts of this paper. I am also thankful to Paolo Ghiggini, Jonathan Hanselman, and Eamonn Tweedy for useful discussions, and to the referee to helpful comments and corrections.  
A large part of this work was completed during an informal visit at UQAM in Summer 2013; I thank Steve Boyer and Olivier Collin for their hospitality. 

\section{Background in bordered Floer homology}\label{backgr}
\input{abs_background}

\section{A $\Z/2$ grading for closed Heegaard diagrams}\label{closedgr}

For a closed $3$-manifold $Y$ represented by a Heegaard diagram $\HH = (\Sigma, \balpha, \bbeta, z)$,  a relative $\Z/2$ grading  can be defined on $\cfhat(\HH)$  by placing two generators in the same grading if the corresponding intersection points in $\ta\cap \tb\subset \mathrm{Sym}^g(\Sigma)$ have the same sign \cite{osz14}.

In \cite[Section 2.4]{dsfh},  Friedl, Juh\'asz, and Rasmussen describe a relative $\Z/2$ grading for sutured Floer homology in terms of intersection signs of $\alpha$ and $\beta$ curves, just like the one for closed manifolds from \cite{osz14}. When $Y$ is a closed, oriented $3$-manifold, then $Y(1)$ denotes the sutured manifold $Y\setminus\Int(B^3)$ with suture an oriented simple closed curve on $\bdy B^3$. Closed Heegaard diagrams for $Y$ and sutured Heegaard diagrams for $Y(1)$ are basically the same (the latter are obtained from the former by removing a neighborhood of the basepoint $z$), 
and the homologies  $\sfh(Y(1))$ and $\hfhat(Y)$ are isomorphic as relatively graded groups.

Below, we recall  the discussion in  \cite[Section 2.4]{dsfh}, modifying it to fit the special case of closed $3$-manifolds. We also expand it with a couple of additional observations. Some of these observations may be implied in \cite{dsfh}, but we state them for the sake of completeness.

Given a closed Heegaard diagram $\HH = (\Sigma, \balpha, \bbeta, z)$ for a $3$-manifold $Y$, choosing an orientation for $\ta$ is the same as choosing a generator for $\Lambda^g(A)$, where $A$ is the subspace of $H_1(\Sigma; \R)$ spanned by the set $\balpha$. Similarly,  choosing an orientation for $\tb$ is the same as choosing a generator for $\Lambda^g(B)$, where $B$ is the subspace of $H_1(\Sigma; \R)$ spanned by $\bbeta$. Since $\mathrm{Sym}^g(\Sigma)$ inherits an orientation from $\Sigma$ viewed as the boundary of the $\alpha$-handlebody,  fixing the sign of intersection of $\ta$ and $\tb$, i.e. turing the relative $\Z/2$ grading into an absolute one, is the same as orienting $\ta$ and $\tb$ relative to each other. This is the same as orienting the tensor product $\Lambda^g(A)\otimes\Lambda^g(B)$. It turns out this is equivalent to orienting the homology of $Y$, see \cite[Definition 2.6]{dsfh}. 

We explain this last claim in a bit more detail, following \cite{dsfh}. The diagram $\HH$ specifies a handle decomposition for $Y$ with one $0$-handle, $1$-handles $A_1, \ldots, A_g$ with belt circles $\alpha_1, \ldots, \alpha_g$, $2$-handles $B_1, \ldots, B_g$ with attaching circles $\beta_1, \ldots, \beta_g$, and a $3$-handle. The Heegaard surface $\Sigma$ is the boundary of the union of  the $0$-handle and the $1$-handles. Let $C_{\ast} = C_{\ast}(Y; \R)$ be the handle homology complex for this handle decomposition (one needs to pick orientations for the cores of the handles in order to read intersections with sign and obtain real coefficients). An orientation $\omega$ of $H_1(Y; \R)\oplus H_2(Y; \R)$ determines an orientation $\omega'$ of $C_1\oplus C_2$ as follows. Choose an ordered basis $h^1_1, \ldots, h^1_n, h^2_1, \ldots, h^2_n$ for $H_1(Y; \R)\oplus H_2(Y; \R)$ compatible with $\omega$,   with $h^i_j\in H_i(Y; \R)$. Pick chains $c^i_j$  representing $h^i_j$. Extend $c^2_1, \ldots, c^2_n$ to a basis $c^2_1, \ldots, c^2_b, b_1, \ldots, b_{g-n}$ for $C_2$. Then 
$$c^1_1, \ldots, c^1_n, \bdy b_1, \ldots, \bdy b_{g-n}, c^2_1, \ldots, c^2_b, b_1, \ldots, b_{g-n}$$
is an oriented basis for $C_1\oplus C_2$. The corresponding orientation $\omega'$ does not depend on the choice of $c^i_j$ and $b_k$. 

Suppose $A_1, \ldots, A_g, B_1, \ldots, B_g$ is an ordered basis of $C_1\oplus C_2$ compatible with $\omega'$ (here we think of $A_i$ and $B_i$ as handles with oriented cores). The corresponding ordering and orientation $\alpha_1, \ldots, \alpha_g, \beta_1, \ldots, \beta_g$ of the belt circles and attaching circles induces an orientation $o(\omega)$ on  $\Lambda^g(A)\otimes\Lambda^g(B)$ by choosing as generator  the wedge product $[\alpha_1]\wedge \ldots\wedge[\alpha_g]\wedge[\beta_1]\wedge\ldots\wedge  [\beta_g]$  of the corresponding homology classes in $H_1(\Sigma; \R)$. It is not hard to see that $o(-\omega) = -o(\omega)$. We make a couple of further observations.

\begin{claim}\label{closed_reor}
Let $\mathfrak o= (\alpha_1, \ldots, \alpha_g, \beta_1,\ldots,  \beta_g)$ and
 $\mathfrak o' = (\alpha_1', \ldots, \alpha_g', \beta_1', \ldots, \beta_g')$ be two choices of ordering and orienting the two sets $\balpha$ and $\bbeta$. Suppose the choices of how to order $\balpha$ and $\bbeta$ differ by permutations $\sigma_{\alpha}$ and $\sigma_{\beta}$, respectively. Let $n_{\alpha}$ and $n_{\beta}$ be the number of $\alpha$ circles, respectively $\beta$ circles, that have  opposite orientations in $\mathfrak o$ and $\mathfrak o'$. If
 $$\mathrm{sign}(\sigma_{\alpha})\mathrm{sign}(\sigma_{\beta})(-1)^{n_{\alpha}}(-1)^{n_{\beta}} = 1,$$
 then the two choices induce the same orientation on $\Lambda^g(A)\otimes\Lambda^g(B)$. Otherwise, i.e. if the product is $-1$, the two choices induce opposite orientations on $\Lambda^g(A)\otimes\Lambda^g(B)$.
\end{claim}
\begin{proof}
This follows directly from the anti-commutativity of the wedge product.
\end{proof}

One can compute the local intersection sign of $\ta$ and $\tb$ from the Heegaard diagram in the following way. Suppose $\mathfrak o= (\alpha_1, \ldots, \alpha_g, \beta_1,\ldots,  \beta_g)$ is an ordering and orientation of the $\alpha$ and $\beta$ circles. Suppose $\xxx$  is a generator of the Heegaard diagram, and write $\xxx = (x_1, \ldots, x_g)$ with $x_i\in \alpha_i$.   Let $\sigma_{\xxx}$ be the permutation for which  $x_i \in \alpha_i\cap \beta_{\sigma_{\xxx}(i)}$.

\begin{claim}
The local intersection sign of $\ta\cap \tb$ at the generator $\xxx$ with respect to the orientation on $\Lambda^g(A)\otimes\Lambda^g(B)$ induced by $\mathfrak o$ can be computed by the formula
$$s(\xxx) = \mathrm{sign}(\sigma_{\xxx})\prod_{i=1}^g s(x_i)$$
\end{claim}
\begin{proof}
This is just \cite[Lemma 2.8]{dsfh} specialized to the case of a closed Heegaard diagram.
\end{proof}

Poincar\'e duality specifies a canonical homology orientation $\omega_{\mathrm{PD}}$ on $H_1(Y; \R)\oplus H_2(Y;  \R)$ as follows. Let $b_1, \ldots, b_g$ be any basis for $H_1(Y; \R)$, and let $b_1^{\ast}, \ldots, b_g^{\ast}$ be the dual basis for $H_2(Y; \R)$, i.e. the basis for  which $b_i\cdot b_j^{\ast} = \delta_{ij}$. Then $\omega_{\mathrm{PD}}$ is the orientation given by the basis $b_1, \ldots, b_g, b_1^{\ast}, \ldots, b_g^{\ast}$. This corresponds to a canonical orientation $o_{\mathrm{can}}:=o(\omega_{\mathrm{PD}})$ on $\Lambda^g(A)\otimes\Lambda^g(B)$.

The above discussion is suited for manifolds that are not rational homology spheres.
In the case of  a homology sphere, $H_1(Y; \R)\oplus H_2(Y; \R)$ is zero-dimensional, but since we study Heegaard diagrams of genus $1$ or higher, $C_1\oplus C_2$ is not. There is still a canonical orientation on  $C_1\oplus C_2$ given by picking any basis $b_1, \ldots, b_g$ for $C_2$ and taking the oriented basis 
$$\bdy b_1, \ldots, \bdy b_g, b_1, \ldots, b_g$$
for $C_1\oplus C_2$. This  corresponds to a canonical orientation $o_{\mathrm{can}}$ on $\Lambda^g(A)\otimes\Lambda^g(B)$.

The relative $\Z/2$ grading on $\cfhat(\HH)$ can be turned into an absolute grading by taking the canonical orientation $o_{\mathrm{can}}$ on $\Lambda^g(A)\otimes\Lambda^g(B)$, and defining the grading of a generator $\xxx$ to be the number $m(\xxx)$ such that $(-1)^{m(\xxx)}$ is the intersection sign $s(\xxx)$ of $\ta$ and $\tb$ at $\xx$. 

Note that our grading convention differs from the one in \cite[Definition 2.9]{dsfh} by a factor of $(-1)^{b_1(Y)}$. Thus, Remark 2.10 of \cite{dsfh} implies that our  $\Z/2$ grading above agrees with the $\Z/2$ grading from \cite{osz14}. Just for fun, we prove this remark in the case of homology spheres. 

\begin{proposition} \label{gr_matrix}
Suppose $H_1(Y; \Z)$ is finite, and let $\HH$ be a Heegaard diagram for $Y$ of genus $g$. Then an ordering and orientation  $\mathfrak o= (\alpha_1, \ldots, \alpha_g, \beta_1,\ldots,  \beta_g)$ of the curves on $\HH$ is compatible with the canonical orientation $o_{\mathrm{can}}$ on $\Lambda^g(A)\otimes\Lambda^g(B)$ if and only if the intersection matrix  with entries $m_{ij} =\alpha_i\cap \beta_j$ has positive determinant.
\end{proposition}

\begin{proof}
The canonical orientation $o_{\mathrm{can}}$ on $\Lambda^g(A)\otimes \Lambda^g(B)$  is given by picking a basis $b_1, \ldots, b_g$ for $C_2$, and preceding it with the basis $\bdy b_1, \ldots, \bdy b_g$ for $C_1$ to obtain an orientation for $C_1\oplus C_2$. In other words, we pick a basis so that the boundary map $\bdy:C_2\to C_1$ is the identity. Then an ordered basis of handles with oriented cores $A_1, \ldots, A_g, B_1, \ldots, B_g$   for  $C_1\oplus C_2$ is compatible with this orientation exactly when the  matrix for the boundary map with respect to this basis has positive determinant. Let $n_{ij}$ be the entries of this boundary matrix. This means that the attaching circle $\beta_j$ for the handle $B_j$ runs $n_{ij}$ times along $A_i$, so the cocore $\alpha_i$ of $A_i$ intersects $\beta_j$ with multiplicity $n_{ij}$. Thus, $m_{ij} = n_{ij}$, i.e. the ordering and orientation $\mathfrak o= (\alpha_1, \ldots, \alpha_g, \beta_1,\ldots,  \beta_g)$ corresponding to the ordered basis $A_1, \ldots, A_g, B_1, \ldots, B_g$ is compatible with  $o_{\mathrm{can}}$  exactly when the intersection matrix with respect to $\mathfrak o$ has positive determinant.
\end{proof}

\begin{corollary}
For homology spheres, the absolute grading defined in this section agrees with the grading defined by Ozsv\'ath and Szab\'o in \cite{osz14}.
\end{corollary}
\begin{proof}
Recall that the grading from \cite{osz14} is defined by requiring that 
$$\chi(\hfhat(Y)) = |H_1(Y; \Z)|.$$
But $\chi(\hfhat(Y))$ is the sum of the gradings of all generators, i.e. $ \#(\ta\cap\tb)$. In other words, the grading from \cite{osz14} is defined by orienting the two tori $\ta$ and $\tb$ so that their intersection sign $ \#(\ta\cap\tb)$ is positive. 

On the other hand, given any orientation for the  two tori, a compatible ordering and orientation on the $\alpha$ and $\beta$ curves induces an intersection matrix $M$  with entries $m_{ij} =\alpha_i\cap \beta_j$, as in Proposition \ref{gr_matrix}, and by definition  $\det(M) = \#(\ta\cap\tb)$. So the requirement for the absolute grading from this section that  $\det(M)>0$ also translates to requiring that the tori are oriented so that  $ \#(\ta\cap\tb)>0$. 

Thus, the two gradings are the same.
\end{proof}

\section{The relative grading for bordered Heegaard Floer homology}\label{rel}

We review and extend the definition of the relative grading  from \cite{dec}.

\subsection{The grading on the algebra}
Let $\zz$ be a pointed matched circle, and let $k$ be the genus of the surface $F(\zz)$. Given a Heegaard diagram $\HH$ with $\bdy \HH=\zz$, recall that the $4k$ points  $\balpha\cap \zz$ come with an ordering $\lessdot$ induced by the orientation of $\zz\setminus z$ \cite[Section 3.2]{bfh2}. For any $\alpha$-arc $\alpha_i$, label its endpoints  $\alpha_i^-$ and $\alpha_i^+$, so that $\alpha_i^-\lessdot\alpha_i^+$, and 
order the $2k$ $\alpha$-arcs so that $\alpha_1^-\lessdot\alpha_2^-\lessdot\ldots\lessdot\alpha_{2k}^-$. Write the matching  as $M(\alpha_i^-)= M(\alpha_i^+) = i$, so that each idempotent $I(\sss)$ corresponds to the set of $\alpha$-arcs indexed by $\sss\subset [2k]$.  Given a set $\sss \subset [2k]$, let $J(\sss)$ denote the multi-index, i.e. ordered set, $(j_1, \ldots, j_n)$ for which $ j_1<\ldots < j_n$ and $\{j_1, \ldots, j_n\} = \sss$.

We define a grading on the algebra $\az$ by  looking at the diagram for the bimodule $\az$ that  was studied in \cite{auroux} (labeled $(\hat F, \{\tilde \alpha^-_i\}, \{\tilde \alpha^+_i\})$), and in \cite[Section 4]{hfmor} (labeled $\mathrm{AZ}(\zz)$).  Figure \ref{az} is an example of $\mathrm{AZ}(\zz)$ when $\zz$ is the split pointed matched circle of genus $2$. Let  $\bdy_{\alpha}\textrm{AZ}(\zz)$ denote the boundary component of $\mathrm{AZ}(\zz)$ which intersects the $\alpha$-arcs, and let  $\bdy_{\beta}\textrm{AZ}(\zz)$ denote the boundary component of $\mathrm{AZ}(\zz)$ which intersects the $\beta$-arcs. Order the $\alpha$-arcs and label their endpoints as above, i.e. following the orientation of $\bdy_{\alpha}\textrm{AZ}(\zz)\setminus z$, and do the same for the $\beta$-arcs,  i.e. following the orientation of $\bdy_{\beta}\textrm{AZ}(\zz)\setminus z$. For each $i$, orient $\alpha_i$ from $\alpha_i^-$ to $\alpha_i^+$, and  $\beta_i$ from $\beta_i^-$ to $\beta_i^+$. For any point $x\in \balpha\cap \bbeta$, define $s(x)$ to be the intersection sign of $\balpha$ and $\bbeta$ at $x$.  Note that the intersection sign of $\alpha_i$ and $\beta_i$ at the diagonal of the triangle is positive. 

\begin{figure}[h]
 \centering
       \labellist
       \pinlabel $z$ at 138 138
       \pinlabel $\alpha_1^-$ at 155 18
       \pinlabel $\alpha_2^-$ at 155 34
       \pinlabel $\alpha_1^+$ at 155 50
       \pinlabel $\alpha_2^+$ at 155 66
       \pinlabel $\alpha_3^-$ at 155 82
       \pinlabel $\alpha_4^-$ at 155 98
       \pinlabel $\alpha_3^+$ at 155 114
       \pinlabel $\alpha_4^+$ at 155 130
       \endlabellist
       \includegraphics[scale=.95]{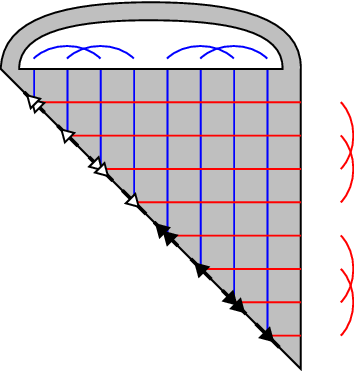} 
       \vskip .2 cm
       \caption{The diagram $\mathrm{AZ}(\zz)$.}
       \label{az}
\end{figure}

Recall that the generators $\mathfrak S(\mathrm{AZ}(\zz))$ are in one-to-one correspondence with the standard generators of $\az$ by strand diagrams. We will denote a generator of $\az$ and the corresponding generator in $\mathfrak S(\mathrm{AZ}(\zz))$ the same way.  Given a generator ${\bf a}$ of $\az$, write its representative in  $\mathfrak S(\mathrm{AZ}(\zz))$ as an ordered subset $\aaa = (x_1, \ldots, x_p)$ of $\balpha\cap\bbeta$, with the intersection points $x_i$ ordered according to the order of the corresponding occupied $\alpha$-arcs.  For a generator $\xxx\in \mathfrak S(\mathrm{AZ}(\zz))$, let  $o_\alpha(\xxx)$ be the set of  $\alpha$-arcs occupied by $\xx$, and let  $o_\beta(\xxx)$ be the set of $\beta$-arcs  occupied by $\xx$.
Define $\sigma_{\xxx}$ to be the permutation for which 
\begin{align*}
x_1 &\in \alpha_{i_1}\cap \beta_{j_{\sigma_{\xxx}(1)}}\\
& \hspace{7pt} \vdots\\
x_p &\in \alpha_{i_p}\cap \beta_{j_{\sigma_{\xxx}(p)}}\\
\end{align*}
where $(i_1, \ldots, i_p) = J(o_\alpha(\xxx))$ and $(j_1, \ldots, j_p) = J(o_\beta(\xxx))$. In other words, $\sigma_{\xxx}$ is the permutation arising from the induced orders on the two sets of occupied arcs.
Define the sign of $\xxx$ by 
$$s(\xxx) = \textrm{sign}(\sigma_{\xxx})\prod_{i=1}^p s(x_i).$$

The following is a slight modification of \cite[Lemma 19]{dec}. 

\begin{proposition}
The sign assigment $s$ induces a differential grading $m$ on $\az$, viewed as a left-right $\az$-$\az$-bimodule. More precisely, define $m: \mathfrak S(\mathrm{AZ}(\zz))\to \Z/2$ by  $s= (-1)^m$ (we can also think of $m$ as a function on the strand diagram generators of $\az$ via the identification with the generators of $\mathrm{AZ}(\zz)$). Then for any $\xx\in \mathfrak S(\mathrm{AZ}(\zz))$ and any generator $\aaa\in \az$ 
\begin{enumerate}
\item $m(\xx \aaa) = m(\xx) + m(\aaa)$, 
\item $m(\aaa \xx) = m(\aaa) + m(\xx)$, and
\item $m(\bdy \xx) = m(\xx)-1$.
\end{enumerate}
\end{proposition}

\begin{proof} In the proof of \cite[Lemma 19]{dec} we verified that the grading respects the right multiplication and the differential, by analyzing sets of half-strips with boundary on  $\bdy_{\alpha}\textrm{AZ}(\zz)$, and rectangles in the interior of $\mathrm{AZ}(\zz)$, respectively. Here, we also need to check multiplication on the left, by looking at half-strips with boundary on $\bdy_{\beta}\textrm{AZ}(\zz)$. The argument from the proof of \cite[Lemma 19]{dec} can be repeated almost verbatim, exchanging the $\alpha$ and $\beta$ labels.
\end{proof}

\subsection{Type $D$ structures} \label{dgr}

Next, we define a $\Z/2$ grading on $\cfdhat$, both for $\alpha$-bordered and $\beta$-bordered Heegaard diagrams. The former are the Heegaard diagrams we recalled in Section \ref{cfacfd_background}; for the latter, see \cite[Section 3.2]{hfmor}.

Given an $\alpha$-bordered Heegaard diagram $\hha$ with $\bdy\hha = -\zz$, order  the $\alpha$-arcs as above,   according to the orientation on $-\bdy \hha$ and starting at the basepoint, and orient them from $\alpha_i^+$ to $\alpha_i^-$. Also order and orient the $\alpha$ and $\beta$ circles, and define a complete ordering on all $\alpha$-curves by $\alpha_1, \ldots, \alpha_{2k}, \alpha^c_1, \ldots, \alpha^c_{g-k}$. Write generators as ordered tuples $\xxx = (x_1, \ldots, x_g)$ to agree with the ordering of the occupied $\alpha$-curves, and for any generator $\xxx$,  define $\sigma_{\xxx}$ to be the permutation for which 
\begin{align*}
x_1 &\in \alpha_{i_1}\cap \beta_{\sigma_{\xxx}(1)}\\
& \hspace{7pt} \vdots\\
x_k &\in \alpha_{i_k}\cap \beta_{\sigma_{\xxx}(k)}\\
x_{k+1} &\in \alpha_1^c\cap \beta_{\sigma_{\xxx}(k+1)}\\
&\hspace{7pt}  \vdots\\
 x_g &\in \alpha_{g-k}^c\cap \beta_{\sigma_{\xxx}(g)},
\end{align*}
where $(i_1, \ldots, i_k) = J(o(\xxx))$.
For any $x_i$, define $s(x_i)$ to be the intersection sign of $\balpha$ and $\bbeta$ at $x_i$.  We also define  $\sigma_\sss$ for each $k$-element set $\sss\subset [2k]$ to be the permutation in $S_{2k}$ that maps the ordered set $(1,\ldots, k)$ to $J(\sss)$ and $(k+1, \ldots, 2k)$ to $J([2k]\setminus \sss)$.
Last, define the sign of $\xxx$ by 
$$s(\xxx) = \textrm{sign}(\sigma_{o(\xxx)})\textrm{sign}(\sigma_{\xxx})\prod_{i=1}^g s(x_i).$$
This description of the sign can be quite frustrating to follow, so we work out a small example in detail at the end of this section. 

Similarly, given a $\beta$-bordered Heegaard diagram $\hhb$ with $\bdy\hhb = -\zz$, order  the $\beta$-arcs as above,   according to the orientation on $-\bdy \hhb$, and orient them from $\beta_i^+$ to $\beta_i^-$. Also order and orient the $\alpha$ and $\beta$ circles, and define a complete ordering on all $\beta$-curves by $\beta^c_1, \ldots, \beta^c_{g-k},\beta_1, \ldots, \beta_{2k}$. Write generators as ordered tuples $\xxx = (x_1, \ldots, x_g)$ to agree with the ordering of the occupied $\beta$-curves, and for any generator $\xxx$,  define $\sigma_{\xxx}$ to be the permutation for which 
\begin{align*}
x_1 &\in \alpha_{\sigma_{\xxx}(1)}\cap \beta_1^c\\
& \hspace{7pt} \vdots\\
x_{g-k} &\in \alpha_{\sigma_{\xxx}(g-k)}\cap \beta_{g-k}^c\\
x_{g-k+1} &\in \alpha_{\sigma_{\xxx}(g-k+1)}\cap \beta_{i_1}\\
&\hspace{7pt}  \vdots\\
 x_g &\in \alpha_{\sigma_{\xxx}(g)}\cap \beta_{i_k},
\end{align*}
where $(i_1, \ldots, i_k) = J(o(\xxx))$.
 Define $s(x_i)$ and $\sigma_{\sss}$ as for $\alpha$-bordered diagrams, and again define the sign of $\xxx$ by 
$$s(\xxx) = \textrm{sign}(\sigma_{o(\xxx)})\textrm{sign}(\sigma_{\xxx})\prod_{i=1}^g s(x_i).$$

\begin{proposition}
The unique function $m: \mathfrak S(\HH)\to \Z/2$  for which $s= (-1)^m$ is a differential grading on $\cfdhat(\HH)$: if $\xx\in \mathfrak S(\HH)$ and $\aaa = a(\brho)\in \cala(-\bdy \HH)$, then $\xx$ and $\aaa$ are homogeneous with respect to the grading, and 
\begin{enumerate}
\item $m(\aaa \xx) =  m(\aaa) + m(\xx)$, and
\item  $m(\bdy \xx) = m(\xx)-1$.
\end{enumerate} 
\end{proposition}

\begin{proof}
 For  $\alpha$-bordered Heegaard diagrams, this was proven in \cite[Lemma 19 and Proposition 20]{dec}. The proof for $\beta$-bordered Heegaard diagrams follows verbatim. 
\end{proof}

 Note that the relative grading is well defined. Ordering and orienting the arcs is uniquely specified by the pointed matched circle, so a $\Z/2$ grading is induced by a choice of ordering and orienting the $\alpha$ and $\beta$ circles, and corresponds to a choice of orientation for $\Lambda^{g-k}(A)\otimes\Lambda^g(B)$ in the case of $\alpha$-bordered diagrams, or $\Lambda^g(A)\otimes\Lambda^{g-k}(B)$ in the case of $\beta$-bordered diagrams . Here $A$ is the subspace of $H_1(\Sigma; \R)$ spanned by $\balpha$, and $B$ is the subspace of $H_1(\Sigma; \R)$ spanned by the set $\bbeta$, as discussed in Section \ref{closedgr} for the closed case. This corresponds to a choice of orientation on $H_{\ast}(Y; \mathbb R)$, as in \cite[Section 2.4]{dsfh}.

\begin{figure}[h]
 \centering
       \labellist
       \pinlabel $z$ at 420 536
       \pinlabel \textcolor{red}{$\alpha_1^-$} at 50 314
       \pinlabel \textcolor{red}{$\alpha_2^-$} at 50 380
       \pinlabel \textcolor{red}{$\alpha_3^-$} at 50 440
       \pinlabel \textcolor{red}{$\alpha_4^-$} at 50 500
       \pinlabel \textcolor{red}{$\alpha_1^+$} at 50 560
       \pinlabel \textcolor{red}{$\alpha_2^+$} at 50 610
       \pinlabel \textcolor{red}{$\alpha_3^+$} at 50 680
       \pinlabel \textcolor{red}{$\alpha_4^+$} at 50 736
       \pinlabel \textcolor{blue}{$\beta_1$} at 90 352  
       \pinlabel \textcolor{blue}{$\beta_2$} at 156 414
       \pinlabel \textcolor{blue}{$\beta_3$} at 220 470
       \pinlabel \textcolor{blue}{$\beta_4$} at 300 540
       \pinlabel \textcolor{red}{$\alpha^c_2$} at 250 340 
       \pinlabel \textcolor{red}{$\alpha^c_1$} at 380 470  
       \pinlabel $A$ at 119 318  
       \pinlabel \reflectbox{\rotatebox{180}{$A$}} at 120 550  
       \pinlabel $B$ at 186 376  
       \pinlabel \reflectbox{\rotatebox{180}{$B$}} at 186 610  
       \pinlabel $C$ at 254 434  
       \pinlabel \reflectbox{\rotatebox{180}{$C$}} at 254 670  
       \pinlabel $D$ at 320 494  
       \pinlabel \reflectbox{\rotatebox{180}{$D$}} at 320 728  
       \pinlabel $p_1$ at 82 304  
       \pinlabel $p_2$ at 146 362
       \pinlabel $p_3$ at 212 422
       \pinlabel $p_4$ at 280 480
       \pinlabel $r_1$ at 162 304  
       \pinlabel $r_2$ at 230 380
       \pinlabel $q_1$ at 296 422
       \pinlabel $q_2$ at 362 500
       \endlabellist
       \includegraphics[scale=.5]{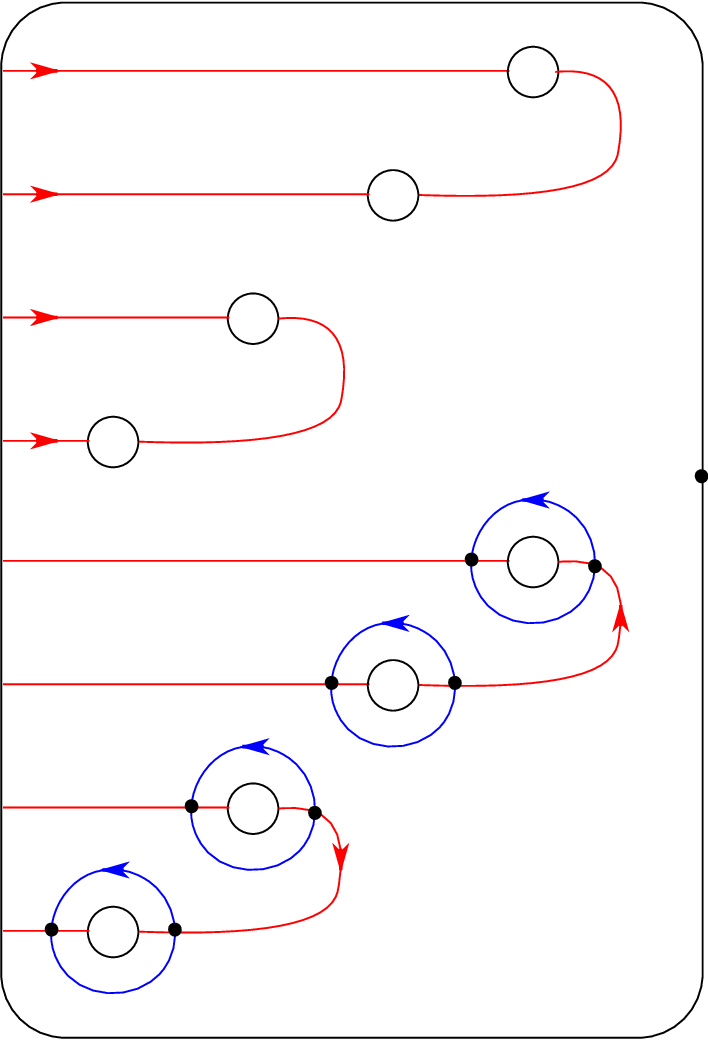} 
       \vskip .2 cm
       \caption{A Heegaard diagram $\HH$ of genus $4$ for a bordered handlebody of genus $2$. The boundary $\bdy\HH$ is the antipodal pointed matched circle.}
       \label{cfdgr}
\end{figure}

\example  Figure \ref{cfdgr} is a Heegaard diagram for a genus $2$ handlebody. The ordering and orientation on the $\alpha$ arcs  is dictated by $\bdy\HH$, and the ordering and orientation on the $\alpha$ and $\beta$ circles was picked arbitrarily. The four generators of the Heegaard diagram, as ordered quadruples according to the order of the $\alpha$-curves, are 
 \begin{align*}
 \xxx &= (p_1,p_3,q_2,r_2),\\
 \yyy &= (p_1, p_4, q_1, r_2),\\
 \zzz &= (p_2, p_3, q_2, r_1),\\ 
 \www &= (p_2, p_4, q_1, r_1).
 \end{align*}
  Since $\xxx$ occupies the arcs $\alpha_1$ and $\alpha_3$,  we have $o(\xxx) = \{1,3\}$, and $\sigma_{o(\xxx)} = \bigl(\begin{smallmatrix}
  1 & 2 & 3 & 4 \\
  1 & 3 & 2 & 4
\end{smallmatrix}\bigr)$. Since $p_1\in \alpha_1\cap \beta_1$, $p_3\in \alpha_3\cap \beta_3$, $q_2\in \alpha^c_1\cap \beta_4$, and $r_2\in \alpha^c_2\cap \beta_2$, we have  $\sigma_{\xxx} = \bigl(\begin{smallmatrix}
  1 & 2 & 3 & 4 \\
  1 & 3 & 4 & 2
\end{smallmatrix}\bigr)$. Then 
$$s(\xxx) = \textrm{sign}(\sigma_{o(\xxx)})\textrm{sign}(\sigma_{\xxx}) s(p_1)s(p_3)s(q_2)s(r_2) = (-1)\cdot 1 \cdot 1\cdot 1\cdot (-1)\cdot 1 = 1.$$
One can compute the signs of the remaining generators similarly. We list the complete data for all four generators below, using the notation ${\bf g}= (g_1, g_2, g_3, g_4)$ that follows the order of the $\alpha$-curves. 
\vspace{.5cm}

\begin{center}
  \begin{tabular}{| c || c | r | c | c | c | c | c |}
    \hline
${\bf g}$
    & $o({\bf g})$ & $\sigma_{o({\bf g})}$ & $\sigma_{\bf g}$ & $\mathrm{sign}(\sigma_{o({\bf g})})$ & $\mathrm{sign}(\sigma_{\bf g})$ & $s(g_1)s(g_2)s(g_3)s(g_4)$ & s({\bf g})\\ \hline \hline
     $\xxx$ & $\{1,3\}$ & $\bigl(\begin{smallmatrix}
  1 & 2 & 3 & 4 \\
  1 & 3 & 2 & 4
\end{smallmatrix}\bigr)$ & $\bigl(\begin{smallmatrix}
  1 & 2 & 3 & 4  \\
  1 & 3 & 4 & 2 
\end{smallmatrix}\bigr)$ & $-1$ & $1$ & $1\cdot 1\cdot (-1)\cdot 1$  & $1$ \\ \hline
    $\yyy$ &  $\{1,4\}$ & $\bigl(\begin{smallmatrix}
  1 & 2 & 3 & 4 \\
  1 & 4 & 2 & 3
\end{smallmatrix}\bigr)$ & $\bigl(\begin{smallmatrix}
  1 & 2 & 3 & 4 \\
  1 & 4 & 3 & 2
\end{smallmatrix}\bigr)$ & $1$ & $-1$ & $1\cdot 1\cdot 1\cdot 1$ & $-1$\\ \hline
    $\zzz$ &  $\{2,3\}$ & $\bigl(\begin{smallmatrix}
  1 & 2 & 3 & 4 \\
  2 & 3 & 1 & 4
\end{smallmatrix}\bigr)$& $\bigl(\begin{smallmatrix}
  1 & 2 & 3 & 4 \\
  2 & 3 & 4 & 1
\end{smallmatrix}\bigr)$ & $1$ & $-1$ &  $1\cdot 1\cdot (-1)\cdot (-1)$  & $-1$ \\ \hline
   $\www$ &  $\{2,4\}$ & $\bigl(\begin{smallmatrix}
  1 & 2 & 3 & 4 \\
  2 & 4 & 1 & 3
\end{smallmatrix}\bigr)$ & $\bigl(\begin{smallmatrix}
  1 & 2 & 3 & 4 \\
  2 & 4 & 3 & 1
\end{smallmatrix}\bigr)$ & $-1$ & $1$  & $1\cdot 1\cdot 1\cdot (-1)$   & $1$ \\
    \hline
  \end{tabular}
\end{center}

\subsection{Type $A$ structures}\label{agr}
A relative grading on  $\cfahat$ can be defined analogously. However, we  modify the definition for $\cfdhat$ slightly,  so that after gluing two bordered Heegaard diagrams, the relative grading $m(\xxx\otimes\yyy)$ for a generator of the resulting closed diagram can be computed  in the simplest way, as the sum $m(\xxx)+ m(\yyy)$.

Given an $\alpha$-bordered Heegaard diagram $\hha$ with $\bdy\hha = \zz$, order  the $\alpha$-arcs as for $\mathrm{AZ}(\zz)$. Also order and orient the $\alpha$ and $\beta$ circles, and define a complete ordering on all $\alpha$-curves by $\alpha^c_1, \ldots, \alpha^c_{g-k}, \alpha_1, \ldots, \alpha_{2k}$. Write generators as ordered tuples $\xxx = (x_1, \ldots, x_g)$ to agree with the ordering of the occupied $\alpha$-curves, and for any generator $\xxx$,  define $\sigma_{\xxx}$ to be the permutation for which 
\begin{align*}
x_1 &\in \alpha^c_{1}\cap \beta_{\sigma_{\xxx}(1)}\\
& \hspace{7pt} \vdots\\
x_{g-k} &\in \alpha^c_{g-k}\cap \beta_{\sigma_{\xxx}(g-k)}\\
x_{g-k+1} &\in \alpha_{i_1}\cap \beta_{\sigma_{\xxx}(g-k+1)}\\
&\hspace{7pt}  \vdots\\
 x_g &\in \alpha_{i_k}\cap \beta_{\sigma_{\xxx}(g)},
\end{align*}
where $(i_1, \ldots, i_k) = J(o(\xxx))$.
For any $x_i$, define $s(x_i)$ to be the intersection sign of $\balpha$ and $\bbeta$ at $x_i$. Last, define the sign of $\xxx$ by 
$$s(\xxx) = \textrm{sign}(\sigma_{\xxx})\prod_{i=1}^g s(x_i).$$

Similarly, one can define a type $A$ grading for $\beta$-bordered diagrams, This is analogous to the type $D$ case, but again ordering the $\alpha$-arcs after the $\alpha$-cirles. 

This sign function induces a grading on $\cfahat(\HH)$ that is compatible with the $\ainf$ operations. The proof of this is similar to the proof for $\cfdhat$.  We choose not to include it in this paper, and phrase all our results in terms of $\cfdhat$.  

\subsection{Pairing}
We show how to relate the relative grading for bordered Floer homology to  the relative grading for Heegaard Floer homology. 

\begin{proposition}\label{grpairing}
Suppose $\HH$ and $\HH'$ are bordered Heegaard diagrams of genus $g$ and $g'$ respectively, such that $\bdy\HH = \zz = -\bdy \HH'$.  Suppose $\xxx$ and $\yyy$ are generators of $\HH$, and $\xxx'$ and $\yyy'$ are generators of $\HH'$, so that $\xxx\boxtimes \xxx'$ and $\yyy\boxtimes \yyy'$ are generators of $\HH \cup \HH'$. The relative grading on 
$$\cfahat(\HH_1)\boxtimes\cfdhat(\HH_2)\cong \cfhat(\HH_1\cup \HH_2)$$
can be computed from the relative grading on bordered Floer homology by 
$$m(\xxx\boxtimes \xxx') - m(\yyy\boxtimes \yyy')  = [m(\xxx) + m(\xxx')] -[m(\yyy) + m(\yyy')].$$
\end{proposition}

\begin{proof}
Let $\alpha^c_1, \ldots, \alpha^c_{g-k}, \alpha_1, \ldots, \alpha_{2k}, \beta_1, \ldots, \beta_g$ be an orientation and ordering on the curves of $\HH$, according to the type $A$ conventions we developed in this section, and let $\alpha_1', \ldots, \alpha_{2k}', {\alpha'}^c_1, \ldots, {\alpha'}^c_{g'-k}, \beta_1', \ldots, \beta_{g'}'$ be an orientation and ordering on the curves on $\HH'$, according to the type $D$ conventions. The orientations on the pairs of arcs $\alpha_i$ and $\alpha'_i$ are compatible, and induce orientations on the closed circles $\widetilde\alpha_i = \alpha_i\cup \alpha_i'$ in $\HH\cup\HH'$. Taking the chosen orientations and concatenating the ordering, we get a total orientation and ordering on the curves in the closed diagram
$$\omega = (\alpha^c_1, \ldots, \alpha^c_{g-k}, \widetilde\alpha_1, \ldots, \widetilde\alpha_{2k},  {\alpha'}^c_1, \ldots, {\alpha'}^c_{g'-k},  \beta_1, \ldots, \beta_g, \beta_1', \ldots, \beta_{g'}').$$
Given generators $\xxx\in \mathfrak S(\HH)$ and $\xxx'\in \mathfrak S(\HH')$, one can verify that 
$$\mathrm{sign}(\sigma_{\xxx\boxtimes\xxx'}) = \mathrm{sign}(\sigma_{\xxx})\mathrm{sign}(\sigma_{o(\xxx')}) \mathrm{sign}(\sigma_{\xxx'}),$$ so 
the product of the signs is
\begin{align*}
s(\xxx)s(\xxx') &=  \textrm{sign}(\sigma_{\xxx})\prod_{i=1}^g s(x_i) \textrm{sign}(\sigma_{o(\xxx')})\textrm{sign}(\sigma_{\xxx'})\prod_{i=1}^{g'} s(x_i) \\
&= \mathrm{sign}(\sigma_{\xxx\boxtimes\xxx'})\left(\prod_{i=1}^g s(x_i)\prod_{i=1}^{g'} s(x_i)\right)\\
&= s(\xxx\boxtimes\xxx').
\end{align*}
All the  $\Z/2$ gradings discussed here are defined by $s = (-1)^m$, so  multiplicativity of signs is equivalent to additivity of gradings. The statement for relative gradings follows.
\end{proof}

\remark  The construction  in this section can easily be extended to the various bordered bimodules from \cite{bimod}, and to the bordered sutured structures developed by  Zarev in \cite{bs}.

\remark To conclude this section, we explain how to relate our $\Z/2$ grading  to the grading from \cite[Section 3]{dec}. 
Recall $I(\sss)$ is in $I(\zz,0) = \mathcal I(\zz)\cap \cala(4k, k)$ whenever $|\sss|=k$. Given such $\sss$, look at $J(\sss) = (s_1, \ldots, s_k)$,  let $\rho_i^{\sss}$ be the Reeb chord from $\alpha_i^-$ to $\alpha_{s_i}^-$ whenever $i\neq s_i$,  and let $\brho^{\sss}$ be the set of all such Reeb chords. Choose grading refinement data by choosing the base idempotent $\sss_0:=\{1, \ldots, k\}$  and defining $\psi(\sss):= gr'(a(\brho^{\sss})) = (\iota(a(\brho^{\sss})); [\brho^{\sss}])$ for every other $\sss$. This specifies a refined grading $gr$ on $\cala(\zz, 0)$. One can verify that the resulting $\Z/2$ grading $m = f\circ gr$ obtained by composing with the map $f$ from \cite[Section 3]{dec} agrees with $s$.  In other words, given $\aaa\in \cala(\zz, 0)$, then $s(\aaa) = (-1)^{m(\aaa)}$, and, for an appropriate choice of a base generator for $\cfdhat(\hha)$ in each $\mathrm{spin}^c$ structure, $ s(\xxx)= (-1)^{m(\xxx)}$.

\section{From relative to absolute gradings}\label{abssec}

For a closed $3$-manifold $Y$, the Poincar\'e duality for homology, specifically the isomorphism between $H_1(Y; \R)$ and $H_2(Y; \R)$,  specifies a canonical orientation on the vector space $H_1(Y; \R)\oplus H_2(Y; \R)$, and hence a canonical $\Z/2$ grading on $\hfhat(Y)$, see Section \ref{closedgr}  and \cite[Section 2.4]{dsfh}.  Alternatively, Ozsv\'ath and Szab\'o define an absolute $\Z/2$ grading by looking at a map from the twisted Heegaard Floer homology $\underline{\hf}^{\infty}(Y)$ to $\hfhat(Y)$, see \cite[Section 10.4]{osz14}. The two gradings from \cite{dsfh} and \cite{osz14} differ by $(-1)^{b_1(Y)}$ \cite[Remark 2.10]{dsfh}, and the grading in Section \ref{closedgr} was defined to agree with the one from \cite{osz14}.
However, when $Y$ has boundary of genus $\geq 1$, the spaces $H_1(Y, \bdy Y; \R)$ and $H_2(Y, \bdy Y; \R)$ are not isomorphic, and there is  no developed bordered analogue of $\underline{\hf}^{\infty}$ either,  and thus there is no analogous way to choose a canonical grading for a manifold with boundary. 

 However, the additional parametrization information for the boundary still allows one to define an absolute grading when we have a \emph{bordered} $3$-manifold $Y = (Y, -\zz, \phi)$, by choosing a special element of $H_1(F(\zz); \R)$. We provide our construction in the remaining part of this section.

Given $\zz$, recall the ordering on the $\alpha$-arcs for $AZ(\zz)$ from Section \ref{rel}. Let $[\alpha_i]$ denote the generator for $H_1(F(\zz); \Z)$ corresponding to the $1$-handle attached to $\alpha_i$ in the construction of $F(\zz)$. We fix the convention  that  the core of the $1$-handle is oriented from $\alpha_i^-$ to $\alpha_i^+$, and closed off inside the $0$-handle for $F(\zz)$ to obtain an oriented circle. Then $[\alpha_i]$ is the homology class of this circle.

Identify $H_1(F(\zz); \Z)\cong \Z\left<[\alpha_1] , \ldots, [\alpha_{2k}] \right>$ with $\Z^{2k}$
 via $[\alpha_i]\mapsto e_i$, where $e_1, \ldots, e_{2k}$ is the standard basis. Since $[\alpha_1], \ldots, [\alpha_{2k}]$ also generate $H_1(F(\zz); \R)$, we can identify $H_1(F(\zz); \R)$ with $\R^{2k}$ in the same way, and view $H_1(F(\zz); \Z)$ as the integer lattice in $H_1(F(\zz); \R)$ under this identification.
Observe that  the real homology $H_1(F(\zz); \R)$ with its intersection form is a symplectic vector space. We define an ordering on the subsets of $H_1(F(\zz); \Z)$ of size $k$ that span Lagrangian subspaces of  $H_1(F(\zz); \R)$. Define a total ordering on $H_1(F(\zz); \Z)$ by 
 \begin{displaymath}
v<^{\ast} w \textrm{ iff } \left\{ \begin{array}{ll}
 |v|<|w|& \textrm{or}\\
|v| = |w| \textrm{ and } v <_{lex} w,
\end{array} \right.
\end{displaymath}
where $|\cdot |$ is the standard norm on $\Z^{2k}\subset \R^{2k}$, and $<_{lex}$ is the lexicographical ordering on $\Z^{2k}$ with respect to the standard basis. In other words, vectors in $H_1(F(\zz); \Z)\cong \Z^{2k}$ are ordered first by their length,  and within given length, the finite number of vectors of this length are ordered as words in the alphabet $\Z$, where letters are ordered  according to their ordering as integers.

This ordering $<^{\ast}$ induces a lexicographical ordering $<^{\ast}_{\mathrm{lex}}$ on the set $S$ of ordered subsets of $H_1(F(\zz); \Z)$ of size $k$. Note that the set $(H_1(F(\zz); \Z),<^{\ast})$, which acts as the ``alphabet" for $S$,  is isomorphic to $(\mathbb N, <)$ as an ordered set, so it is well-ordered. The set $(S, <^{\ast}_{\mathrm{lex}})$ then is isomorphic to $\mathbb N^k$ with the lexicographical ordering induced by $<$, so $(S, <^{\ast}_{\mathrm{lex}})$ is well-ordered as well.  

Define the subset 
$$L:=\{l\in S| l \textrm{ spans a Lagrangian subspace of } H_1(F(\zz); \R)\}.$$
We say that $l \in L$ is \emph{embeddable}, if $l$ can be represented by $k$ disjoint, embedded, oriented circles $c_1, \ldots, c_k$ on $F$, i.e. if we can find such embedded circles, so that $l = ([c_1], \ldots, [c_k])$. Note that not every element in $L$ is embeddable. For example, it is easy to see that for any $l\in L$, $2l$ is not embeddable. We define
$$L^{\mathrm{emb}} :=\{l\in L | l \textrm{ is embeddable}\}.$$
\begin{claim}
The set $L^{\mathrm{emb}}$ is non-empty.
\end{claim}
\begin{proof}
One can always find a set of $k$ pairwise disjoint, homologically independent circles on a closed surface of genus $k$, so let $c_1, \ldots, c_k$ be $k$ pairwise disjoint, oriented,  homologically independent circles on $F$. Then each curve $c_i$ specifies a class $[c_i]\in H_1(F(\zz); \Z)$, and since the curves are pairwise disjoint, the homological intersection numbers $[c_i]\cdot [c_j]$ are all zero, so $[c_1], \ldots, [c_k]$ span a Lagrangian subspace in $H_1(F(\zz); \R)$. Thus, 
 $([c_1], \ldots, [c_k])\in L^{\mathrm{emb}}$.
\end{proof}
Since $(S, <^{\ast}_{\mathrm{lex}})$ is well-ordered, and the subset $L^{\mathrm{emb}}\subset S$ is non-empty, it follows that $L^{\mathrm{emb}}$, with the ordering induced from $(S, <^{\ast}_{\mathrm{lex}})$,
 has a (unique) smallest element. Let $l_{\zz} = (l_1, \ldots, l_k)$ be the smallest element in $L^{\mathrm{emb}}$.

We now construct a diagram $\hz$, given a pointed matched circle $\zz = (Z, {\bf a}, M, z)$. We start with $(Z\setminus \textrm{nbd}(z), {\bf a})\times I$, and attach $2$-dimensional $1$-handles to the $0$-spheres $M^{-1}(i)\times \{0\}$ to obtain a compact surface $\Sigma$ of genus $2k$ with one boundary component and $2k$ $\alpha$-arcs. Note that the boundary of the resulting Heegaard diagram is $\zz$, and order and orient the $\alpha$-arcs according to the convention for $\mathrm{AZ}(\zz)$ from Section \ref{rel}. Reinsert the basepoint $z$ in the region of $\bdy \Sigma\setminus {\bf a}$ containing $\bdy Z\times I$. Let $\gamma_i$ be the closure of $\alpha_i$ along $(Z\setminus \textrm{nbd}(z), {\bf a})\times \{1\}$ to a circle, oriented compatibly with $\alpha_i$. Add $k$ pairwise disjoint, oriented $\beta$-circles to represent  $l_{\zz} = (l_1, \ldots, l_k)$, i.e. so that if $l_i = \sum_{j\in [2k]}a_{ij}[\gamma_j]$, then $[\beta_i] =  \sum_{j\in [2k]}a_{ij}[\gamma_j] =l_i$. This is possible, since $l_{\zz}$ is embeddable.
 The resulting diagram $\hz$ specifies a bordered handlebody.

The  ordering and orientation on the $\alpha$ and $\beta$ curves prescribed above induces a $\Z/2$ grading $m$ on the generators of $\hz$, according to the type $A$ conventions from Section \ref{agr}.

Given a bordered Heegaard diagram $\HH$ with boundary $\bdy \HH = -\zz$, we turn the relative grading for $\HH$ from Section \ref{rel} into an absolute grading by requiring that the resulting grading on the generators of  $\hz \cup \HH$ defined by $m(\xxx\otimes \yyy) := m(\xxx)+m(\yyy)$ agrees with the absolute grading from Section \ref{closedgr}.  Note that by Proposition \ref{grpairing}, this ``additive" definition makes sense.

 \begin{proof}[Proof of Theorem \ref{grthm}]
The absolute grading  we just defined agrees with the relative grading from Section \ref{dgr}, and the behavior stated by Equations $(1)$ and $(2)$ of Theorem \ref{grthm}  has been verified in Section \ref{dgr}. It only remains to show that the absolute grading is well defined.

Note that while the Lagrangian $l_{\zz}$ is well-defined, the corresponding Heegaard diagram $\hz$ is not, as we only specified the homology class and orientation of each $\beta$-curve, but not the precise embedding. However,  the homology data that goes into fixing the grading is the same in the following sense. Let   $\hz$ and $\hz'$ be two Heegaard diagrams constructed as above (i.e. two choices of how exactly to place the $\beta$-curves). They  specify  (oriented) bordered handlebodies $H_{\zz}$ and $H_{\zz}'$. 
The corresponding sets of oriented curves $\bbeta = \{\beta_1, \ldots, \beta_k\}$ and $\bbeta' = \{\beta_1', \ldots, \beta_k'\}$ intersect the arcs $\alpha_1, \ldots, \alpha_{2k}$ algebraically in the same way, i.e. $\# (\beta_i\cup \alpha_j) = \# (\beta_i'\cup \alpha_j)$,  since $\bbeta$ and $\bbeta'$ are both  ordered and oriented according to $l_{\zz}$. Also, there are no $\alpha$-circles, and the $\alpha$-arcs are ordered and oriented canonically. Hence, given a bordered Heegaard diagram $\HH$ with boundary $\bdy \HH = -\zz$, an ordering and orientation $\mathfrak o$ on its $\alpha$ and $\beta$ curves, when concatenated with the ordering and orientation $\beta_1, \ldots, \beta_k$ or $\beta_1', \ldots, \beta_k'$, induces the same intersection form for $\hz \cup \HH$ as for $\hz' \cup \HH$. Therefore, $\mathfrak o$ concatenated with  $\beta_1, \ldots, \beta_k$ is compatible with the canonical orientation on $\Lambda^g(A)\otimes\Lambda^g(B)$ from Section \ref{closedgr} if and only if $\mathfrak o$ concatenated with  $\beta_1', \ldots, \beta_k'$ is compatible with the canonical orientation on $\Lambda^g(A')\otimes\Lambda^g(B')$. Here $A$ is the vector space spanned by the $\alpha$-circles on $\hz \cup \HH$, $A'$ is the space spanned by the $\alpha$-circles on $\hz' \cup \HH$, and $B$ and $B'$ are the analogous spaces spanned by the $\beta$-circles. 
Hence, the absolute grading on $\cfdhat(\HH)$ defined in this section does not depend on the choice of $\hz$.
\end{proof}

We work out a complete example below.

\example   Let  $\zz$ be the antipodal matched circle for a surface of genus $2$. Then $H_1(F(\zz); \Z)\cong \Z^4$ has standard basis $e_1, e_2, e_3, e_4$, corresponding to the four $\alpha$-arcs. The ordered set $(H_1(F(\zz); \Z), <^{\ast})$ starts as
\[\vec  0, -e_1, -e_2, -e_3, -e_4, e_4, e_3, e_2, e_1, -e_1-e_2, -e_1-e_3, \ldots \]
Here, the set $S$ consists of ordered pairs of elements in $(H_1(F(\zz); \Z)$. Given  two pairs $(a,b)$ and $(c,d)$ in $S$, 
 \begin{displaymath}
(a,b)<^{\ast}_{\mathrm{lex}} (c,d) \textrm{ iff } \left\{ \begin{array}{ll}
 a<^{\ast}c & \textrm{or}\\
a = c\textrm{ and } b <^{\ast} d.
\end{array} \right.
\end{displaymath}
With this total ordering, the set $S$ begins as
\[(\vec 0, \vec 0), (\vec 0, -e_1),  (\vec 0, -e_2),  (\vec 0, -e_3), (\vec 0, -e_4), (\vec 0, e_4), \ldots\]
 The subset $L\subset S$ consisting of ``Lagrangian bases", with the ordering induced from $<^{\ast}_{\mathrm{lex}}$ begins as
 \[(-e_1, -e_2+e_3),  \ldots\]

Thus, we need to find a Heegaard diagram $\hz$ of genus $2$, with $\bdy\hz = \zz$, $\alpha$-arcs oriented and ordered according to the type $A$ conventions, no $\alpha$-circles, and  oriented circles $\beta_1$ and $\beta_2$ such that $[\beta_1] = -[\gamma_1]$ and $[\beta_2] = -[\gamma_2] + [\gamma_3]$.

 Figure \ref{hz} show one possible diagram $\hz$.

\begin{figure}[h]
 \centering
       \labellist
       \pinlabel $z$ at 130 10
       \pinlabel $x_1$ at 113 100
       \pinlabel $x_2$ at 111 263
       \pinlabel $y_1$ at 145 80
       \pinlabel $y_2$ at 145 120
       \pinlabel $y_3$ at 145 162
       \pinlabel \textcolor{red}{$\alpha_1^-$} at 170 30
       \pinlabel \textcolor{red}{$\alpha_2^-$} at 170 71
       \pinlabel \textcolor{red}{$\alpha_3^-$} at 170 112
       \pinlabel \textcolor{red}{$\alpha_4^-$} at 170 153
       \pinlabel \textcolor{red}{$\alpha_1^+$} at 170 200
       \pinlabel \textcolor{red}{$\alpha_2^+$} at 170 250
       \pinlabel \textcolor{red}{$\alpha_3^+$} at 170 287
       \pinlabel \textcolor{red}{$\alpha_4^+$} at 170 330
       \pinlabel \textcolor{blue}{$\beta_1$} at 142 45
       \pinlabel \textcolor{blue}{$\beta_2$} at 130 240
       \endlabellist
       \includegraphics[scale=.6]{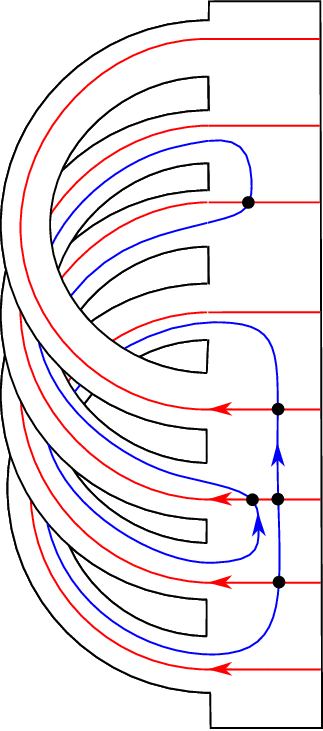} 
       \vskip .2 cm
       \caption{A diagram $\HH_{\zz}$ when $\zz$ is the antipodal pointed matched circle.}
       \label{hz}
\end{figure}

The four generators of $\hz$, as ordered pairs according to the order of the $\alpha$-curves, are $\aaa = (y_1, x_1)$, $\bbb = (x_2, y_2)$, $\ccc = (x_1, y_3)$, and $\ddd = (x_2, y_3)$.  We list the complete data needed for computing their signs below.
\vspace{.5cm}

\begin{center}
  \begin{tabular}{| c || c  | c | c | c | c | c |}
    \hline
generator ${\bf g}$
    & $o({\mathbf{g}})$ &  $\sigma_{\mathbf g}$ &  $\mathrm{sign}(\sigma_{\bf g})$ & $s(g_1)s(g_2)$ & s({\bf g})\\ \hline \hline
     $\aaa$ & $\{2,3\}$ & $\bigl(\begin{smallmatrix}
  1 & 2  \\
  1 & 2 
\end{smallmatrix}\bigr)$ &   $1$ & $(-1)\cdot (-1)$  & $1$ \\ \hline
    $\bbb$ &  $\{2,3\}$ & $\bigl(\begin{smallmatrix}
  1 & 2 \\
 2 & 1
\end{smallmatrix}\bigr)$ & $-1$ & $ 1\cdot (-1)$ & $1$\\ \hline
    $\ccc$ &  $\{3,4\}$ & $\bigl(\begin{smallmatrix}
  1 & 2  \\
  2 & 1
\end{smallmatrix}\bigr)$  & $-1$ &  $ (-1)\cdot (-1)$  & $-1$ \\ \hline
   $\ddd$ &  $\{2,4\}$ & $\bigl(\begin{smallmatrix}
  1 & 2  \\
  2 & 1
\end{smallmatrix}\bigr)$  & $-1$  & $1\cdot (-1)$   & $1$ \\
    \hline
  \end{tabular}
\end{center}
\vspace{.5cm}

We use  $\hz$ to pin down the absolute grading for the Heegaard diagram $\HH$ from Figure \ref{cfdgr}. The generators of the closed Heegaard diagram $\hz\cup \HH$ are $\aaa \boxtimes \yyy$, $\bbb \boxtimes\yyy$, and $\ddd\boxtimes\xx$. The signs induced by the randomly-picked ordering and orientation on the curves for $\HH$ in Figure \ref{cfdgr} are  
\begin{align*}
s(\aaa \boxtimes \yyy) &= s(\aaa)s(\yyy) = -1,\\
s(\bbb \boxtimes\yyy) &=  s(\bbb)s(\yyy) = -1,\\
s(\ddd\boxtimes\xx) &= s(\ddd)s(\xx) = 1.
\end{align*}
In $\Z/2$ notation, the gradings induced by this choice are
\begin{align*}
m(\aaa \boxtimes \yyy) &=  1,\\
m(\bbb \boxtimes\yyy) &=  1,\\
m(\ddd\boxtimes\xx) &=  0.
\end{align*}
We see that the Euler characteristic of $\hfhat$ for the closed manifold $Y$ specified by $\hz\cup \HH$ is non-zero, so $H_1(Y; \Z)$ is finite (in fact, the reader can verify that $Y\cong S^3$). 
In this case, the absolute $\Z/2$ grading  is specified by 
$$\chi(\hfhat(Y)) = |H_1(Y; \Z)|.$$
We need the Euler characteristic to be $1$, and not $-1$, so we need to shift the grading on $\cfdhat(\HH)$ induced by the choices in Figure \ref{cfdgr}.  This could be achieved, for example, by reversing the orientation of $\alpha_2^c$.
\\

We finish this section with a bordered version of Claim \ref{closed_reor}.  For a bordered Heegaard diagram, the ordering and orientation on the arcs is fixed, so one can only make choices regarding the circles.

\begin{claim}\label{bord_reor}
Let $\HH=(\Sigma, \balpha, \bbeta, z)$ be a bordered Heegaard diagram of genus $g$ with $2k$ $\alpha$-arcs. Let $\omega= (\alpha_1^c, \ldots, \alpha_{g-k}^c, \beta_1,\ldots,  \beta_g)$ and
 $\omega' = ({\alpha'}^c_1, \ldots, {\alpha'}^c_{g-k}, \beta_1', \ldots, \beta_g')$ be two choices of ordering and orienting the $\alpha$ and $\beta$ circles. Suppose the choices of how to order the $\alpha$ circles and the $\beta$ circles differ by permutations $\sigma_{\alpha}$ and $\sigma_{\beta}$, respectively. Let $n_{\alpha}$ and $n_{\beta}$ be the number of $\alpha$ circles, respectively $\beta$ circles, that have  opposite orientations in $\omega$ and $\omega'$. If
 $$\mathrm{sign}(\sigma_{\alpha})\mathrm{sign}(\sigma_{\beta})(-1)^{n_{\alpha}}(-1)^{n_{\beta}} = 1,$$
 then the two choices induce the same grading on $\cfdhat(\HH)$. Otherwise, i.e. if the product is $-1$, the two choices induce opposite gradings on $\cfdhat(\HH)$ (one is a shift of the other by $1$).
\end{claim}

\begin{proof}
By concatenating with the canonical ordering and orientation of the curves on $\hz$,  $\omega$ and $\omega'$ induce two choices, $\widetilde\omega$ and $\widetilde\omega'$ respectively, of ordering and orienting the curves on the closed Heegaard diagram $\hz\cup \HH$. Define $\widetilde\sigma_{\alpha}$, $\widetilde\sigma_{\beta}$, $\widetilde n_{\alpha}$, and $\widetilde n_{\beta}$ for the pair $\widetilde\omega$ and $\widetilde\omega'$, in the same way as $\sigma_{\alpha}$, $\sigma_{\beta}$, $n_{\alpha}$, and $n_{\beta}$ were defined for $\omega$ and $\omega'$.  Clearly, $\widetilde n_{\alpha} =  n_{\alpha}$  and $\widetilde n_{\beta} =  n_{\beta}$. The permutations $\widetilde\sigma_{\alpha}$ and $\widetilde\sigma_{\beta}$ are just $\sigma_{\alpha}$ and $\sigma_{\beta}$ extended over the new curves by the identity, so $\mathrm{sign}(\widetilde\sigma_{\alpha}) = \mathrm{sign}(\sigma_{\alpha})$, and $\mathrm{sign}(\widetilde\sigma_{\beta}) = \mathrm{sign}(\sigma_{\beta})$. Then 
 $$  \mathrm{sign}(\widetilde\sigma_{\alpha})\mathrm{sign}(\widetilde\sigma_{\beta}) (-1)^{\widetilde n_{\alpha}} (-1)^{\widetilde n_{\beta}}= \mathrm{sign}(\sigma_{\alpha})\mathrm{sign}(\sigma_{\beta})(-1)^{n_{\alpha}}(-1)^{n_{\beta}}.$$
By Proposition \ref{grpairing}, since the grading on $\hz$ is fixed,  $\widetilde\omega$ and $\widetilde\omega'$ induce the same grading on $\hz\cup\HH$ if and only if $\omega$ and $\omega'$ induce the same grading on $\HH$. The claim follows.
\end{proof}

 Let $A$ be the $(g-k)$-dimensional subspace of $H_1(\Sigma; \R)$ spanned by the $\alpha$ circles, and let $B$ be the $g$-dimensional subspace of $H_1(\Sigma; \R)$ spanned by the $\beta$ circles. By the anticommutativity of  the wedge product, Claim \ref{bord_reor} is equivalent  to the statement that the absolute grading described in this section  is a preferred orientation on  $\Lambda^{g-k}(A)\otimes\Lambda^g(B)$.

\remark Using the same idea, one can obtain absolute gradings for the modules associated to $\beta$-bordered Heegaard diagrams, for the various bimodules from \cite{bimod}, and for the more general structures from \cite{bs}.

\section{Invariance}\label{inv}
This section is the proof of Theorem \ref{invariance}. 

We begin with the statement that curves of index $n$ shift the grading by $n$. In \cite[Section 6.1]{bfh2},  if $\vec{\rho} = (\rho_1, \ldots, \rho_l)$ a non-empty  sequence of Reeb chords on $-\zz$, then $-\vec{\rho}$ is defined to be the sequence $(-\rho_1, \ldots, -\rho_l)$ of the same chords with reversed orientation, and $a(-\vec{\rho})$ is defined as the product $a(-\rho_1) \cdots a(-\rho_l)\in \az$.

\begin{proposition}\label{indgr}
Let $\HH_D$ be a bordered Heegaard diagram, and let $B_D\in \pi_2(\xx_D, \yy_D)$. If $B_D$ is provincial and $\ind(B_D) = n$, then 
$$m(\xx) = m(\yy) + n.$$
Otherwise, if $\vec{\rho}$ is a non-empty  sequence of Reeb chords on $\bdy\mathcal H_D$ for which $(B_D, \vec{\rho})$ is  compatible, and $\ind(B_D, \vec{\rho}) = n$, then
$$m(\xx) = m(a(-\vec{\rho})) + m(\yy) + n.$$
\end{proposition}

When $B_D$ is provincial, the index formula says that $e(B_D) + n_{\xx}(B_D) + n_{\yy}(B_D) = n$, and the result follows (e.g. close off the Heegaard diagram $H_D$ any way you like, and observe that any two generators connected by the domain of $B_D$ have Maslov gradings that differ by $n$ mod $2$). In the second case, the proof follows directly from the following two lemmas.

Let $\HH_A$ and $\HH_D$ be two bordered Heegaard diagrams that glue up to form a closed Heegaard diagram $\mathcal H$. Recall that there is a natural identification of $\pi_2(\xx_A\cup\xx_D, \yy_A\cup \yy_D)$ with the subset of $\pi_2(\xx_A, \yy_A)\times \pi_2(\xx_D, \yy_D)$ consisting of pairs $(B_A, B_D)$ with $\bdy^{\bdy}(B_A) + \bdy^{\bdy}(B_D) = 0$ (and on the level of domains this identification is given by adding the two domains) \cite[Lemma 4.32]{bfh2}. Recall that when $B_A$ and $B_D$ agree along the boundary, $B_A\natural B_D$ denotes the homology class in $\pi_2(\xx_A\cup\xx_D, \yy_A\cup \yy_D)$ identified with $B_A, B_D$. The index is additive under $\natural$ in the following sense.

\begin{lemma} \label{ind}
Let $\HH_A$ and $\HH_D$ be two bordered Heegaard diagrams that glue up to form a closed Heegaard diagram $\mathcal H$. Let $\vec{\rho}$ be a sequence of Reeb chords on $\bdy\mathcal H_D$, and let $B_A\in \pi_2(\xx_A, \yy_A)$ and $B_D\in \pi_2(\xx_D, \yy_D)$ be homology classes such that both pairs $(B_A, \brho)$ and  
 $(B_D, \vec{\rho})$ are compatible, where $\brho$ is the set of Reeb chords for $a(-\vec{\rho})$. 
Then 
$$\ind(B_A\natural B_D) = \ind(B_A) + \ind(B_D) - 1. $$
\end{lemma}

\begin{proof}
Note that since $(B_A, \brho)$ and  $(B_D, \vec{\rho})$ are both compatible, then $B_A$ and $B_D$ agree along the boundary, and the corresponding homology class $B_A\natural B_D$ exists. We also remark  that typically the set $\brho$ for which $a(\brho) = a(-\vec{\rho})$ is not the same as the set of Reeb chords in the sequence $-\vec{\rho}$. 
In general, given a compatible pair $(B, \vec{\brho})$, where $\vec{\brho}$ is a sequence of non-empty sets of Reeb chords,  the \emph{embedded index} of the pair is defined as 
$$\ind (B, \vec{\brho}) = e(B) + n_{\xx}(B) + n_{\yy}(B)+ |\vec{\brho}|+ \iota(\vec{\brho}),$$
see \cite[Definition 5.61]{bfh2}.
In our case, $\vec{\rho}$ is a sequence of one-element sets, whereas $\brho$ is a sequence of length one, i.e. of only one set, so $|\brho|=1$. The  indices in these two cases are given by the formulas
\begin{align*}
\ind (B_A, \brho) &= e(B_A) + n_{\xx_A}(B_A) + n_{\yy_A}(B_A)+ 1+ \iota(\brho)\\
\ind(B_D, \vec{\rho}) &= e(B_D) + n_{\xx_D}(B_D) + n_{\yy_D}(B_D) + |\vec{\rho}|+\iota(\vec{\rho}).
\end{align*}
By \cite[Equation 5.58]{bfh2},
$$\iota(\vec{\rho}) = \sum_i \iota(\rho_i) +  \sum_{i<j}L(\rho_i, \rho_j) =  -\frac{|\vec{\rho}|}{2}+  \sum_{i<j}L_D(\rho_i, \rho_j), $$ 
where $L_D(\rho_i, \rho_j)$ is the linking number $L(\rho_i, \rho_j)$  on $\bdy \HH_D$. By \cite[Lemma 5.60]{bfh2},
$$\iota(\brho) = \iota(a(-\vec{\rho})) = -\frac{|\vec{\rho}|}{2} + \sum_{i<j}L_A(-\rho_i, -\rho_j),$$
where $L_A(-\rho_i, -\rho_j)$ is the linking number $L(-\rho_i, -\rho_j)$ but considered on $\bdy \HH_A = -\bdy \HH_D$. Since $L_A(-\rho_i, -\rho_j) = - L_D(-\rho_i, -\rho_j) =  - L_D(\rho_i, \rho_j)$, we can write
$$\iota(\brho) =  -\frac{|\vec{\rho}|}{2} - \sum_{i<j}L_D(\rho_i, \rho_j) = -|\vec{\rho}| - \iota(\vec{\rho}).$$
Note that $e$, $n_x$, and $n_y$ are additive in this setup, and so we have
\begin{align*}
\ind (B_A\natural B_D) &= e(B_A\natural B_D) + n_{\xx_A\cup \xx_D}(B_A\natural B_D)+ n_{\yy_A\cup \yy_D}(B_A\natural B_D)\\
&= e(B_A) + n_{\xx_A}(B_A) + n_{\yy_A}(B_A) + e(B_D) + n_{\xx_D}(B_D) + n_{\yy_D}(B_D)\\
&= \ind(B_A, \brho) - 1- \iota(\brho) + \ind(B_D, \vec{\rho})  -  |\vec{\rho}|-\iota(\vec{\rho})\\
&=  \ind(B_A, \brho) + \ind(B_D, \vec{\rho}) - 1.\qedhere
\end{align*}
\end{proof}
To complete the proof of Proposition \ref{indgr}, we use the Heegaard diagram $AZ(\zz)$ and a simple gluing argument, along with the following observation.
\begin{lemma} \label{indaz}
Given a set of Reeb chords $\brho$,  if  $B$ is the homology class for the diagram $AZ(\zz)$ that represents the multiplication of $I(\sss)$ by $a(\brho)$, then  $\ind(B) = 1$.
\end{lemma}

\begin{proof} 
Denote $\xx := I(\sss)$ and $\yy := I(\sss)a(\brho)$. The domain of $B$ consists of $|\brho|$ half strips on $AZ(\zz)$. For each half strip $B_i$, let $x_i$ be the bottom left corner, and $y_i$ be the top right corner. Observe that
\begin{displaymath}
n_{x_i}(B_j) = \left\{ \begin{array}{ll}
\frac{1}{4} & \textrm{if $i=j$}\\
0 & \textrm{otherwise,}
\end{array} \right.
\end{displaymath}
so 
$$n_{\xx}(B) = \sum_{i, j}n_{x_i}(B_j)  = \frac{|\brho|}{4},$$
and
\begin{displaymath}
n_{y_i}(B_j) = \left\{ \begin{array}{ll}
\frac{1}{4} & \textrm{if $i=j$}\\
\frac{1}{2} & \textrm{if $\rho_j^+ = \rho_i^-$}\\
1 & \textrm{if  $\rho_j^-  \lessdot \rho_i^- \lessdot \rho_j^+  \lessdot \rho_i^+$}\\
0 & \textrm{otherwise,}
\end{array} \right.
\end{displaymath}
so
\begin{align*}
n_{\yy}(B) &= \sum_{i, j}n_{y_i}(B_j) \\
&= \frac{|\brho|}{4}+ \frac{|\{ \{\rho_i, \rho_j\}\subset \brho| \rho_j^+ = \rho_i^- \}|}{2} + |\{ \{\rho_i, \rho_j\}\subset \brho| \rho_j^-  \lessdot \rho_i^- \lessdot \rho_j^+  \lessdot \rho_i^+\}|\\
&= \frac{|\brho|}{4} + \sum_{\{\rho_i, \rho_j\}\subset \brho}|L(\rho_i, \rho_j)|\\
&= -\iota(\brho) - \frac{|\brho|}{4}.
\end{align*}
Finally, 
\begin{align*}
\ind(B) &= e(B) + n_{\xx}(B) + n_{\yy}(B) + 1 + \iota(\brho)\\
&= 0 + \frac{|\brho|}{4}  -\iota(\brho) - \frac{|\brho|}{4}+ 1+ \iota(\brho) \\
&=  1.\qedhere
\end{align*}
\end{proof}
We note that the assertion of Lemma \ref{indaz} is implied by the statement of \cite[Proposition 4.1]{hfmor}, but we chose to include the proof for completeness.

Now, suppose $B_D\in\pi_2 (\xx, \yy)$, $(B_D, \vec{\rho})$ is compatible, and $\ind(B_D, \vec{\rho}) = n$. Let $B_A$ be the domain in $AZ(\zz)$ from $I_A(\xx)$ to $I_A(\xx) a(-\vec{\rho}) = I_A(\xx) a(\brho)$. Technically, $AZ(\zz)\cup \HH_D$ is not a closed Heegaard diagram, but since $B_A$ is a one-sided domain, the computation in Lemma \ref{ind} still applies, so 
$$\ind(B_A\natural B_D) =  \ind(B_A, \brho) + \ind(B_D, \vec{\rho}) - 1 = n.$$
The orderings and orientations on the curves in $AZ(\zz)$  and $\HH_D$ have already been fixed, and can be concatenated, both for the $\alpha$ curves and for the $\beta$-curves, to produce an ordering and orientation on the curves in $AZ(\zz)\cup \HH_D$. Analogous to Proposition \ref{grpairing}, given generators $\xxx\in AZ(\zz)$ and $\yyy\in \HH_D$ such that $\zzz:=\xxx\boxtimes\yyy$ is a non-zero generator in   $AZ(\zz)\cup \HH_D$, one can compute the quantity $s(\zzz):=\mathrm{sign}(\sigma_{\zzz})\prod_{i=1}^{g+k}z_i$ by 
$$\textrm{sign}(\sigma_{\xxx})\prod_{i=1}^k s(x_i) \textrm{sign}(\sigma_{o(\yyy)})\textrm{sign}(\sigma_{\yyy})\prod_{i=1}^{g} s(y_i) = s(\xxx)s(\yyy),$$
where $s(\xxx)$ is the sign function for $AZ(\zz)$ at $\xxx$, and $s(\yyy)$ is the type $D$ sign function for $\HH_D$ at $\yyy$.

Since $I_A(\xx)\boxtimes \xx$ and $I_A(\xx)a(\brho)\boxtimes\yy$ are connected by a closed domain of index $n$ in $AZ(\zz)\cup \HH_D$, then $s(I_A(\xx)\boxtimes \xx)s(I_A(\xx)a(\brho)\boxtimes\yy) = (-1)^n$ by classical Heegaard Floer homology. Then $s(I_A(\xx))s(\xx)s(I_A(\xx)a(\brho))s(\yyy) = (-1)^n$. But $s(I_A(\xx))=1$, so 
$$s(\xx) = (-1)^n s(I_A(\xx)a(\brho))s(\yyy).$$
In $\Z/2$ grading notation,
\begin{align*}
m(\xx) &= m(a(\brho)) + m(\yy) + n\\
&= m(a(-\vec{\rho})) + m(\yy) + n.
\end{align*}
This completes the proof of  Proposition \ref{indgr}. 

We now return to the invariance proof.

 \begin{proof}[Proof of Theorem \ref{invariance}]
 Most of the proof follows from that of the ungraded version,  \cite[Theorem 6.16]{bfh2}. We only need to verify  that the maps induced by a change of almost complex structure, isotopy, stabilization, and handleslide preserve gradings.
 
 The argument for a change of almost complex structure or isotopy carries verbatim from \cite{bfh2}, along with the observation that the ``continuation map" defined in \cite[Section 6.3.1]{bfh2} counts domains of index $0$, so by Proposition \ref{indgr} it is a graded homotopy equivalence. 
Invariance under isotopies can also be verified by defining continuation maps that count domains of index $0$.

As observed in \cite[Section 6.3.3]{bfh2},  stabilizing the region of $\HH$ containing the basepoint  produces a new Heegaard diagram $\HH'$, and an isomorphic chain complex $\cfdhat(\HH')$. 
 Note that labeling the new $\alpha$-circle and $\beta$-circle with the highest index, and orienting them so that the new intersection point is positive is a choice compatible with the desired grading on the closed manifold $H_{\zz}\cup Y(\HH')$, and for this same choice, the  grading on $\cfdhat(\HH')$ coincides with the grading on $\cfdhat(\HH)$ under the isomorphism.

The cases of the different kinds of handleslides are similar to each other, and   we verify here the hardest one -- handlesliding an $\alpha$-arc over an $\alpha$-circle. Again, most of the argument carries verbatim from \cite{bfh2}, and we only need to check that the triangle map defined in \cite[Equation 6.33]{bfh2} preserves the $\Z/2$ grading.

Let $\HH_D = (\Sigma, \balpha, \bbeta, z)$ be an admissible bordered Heegaard diagram. Order and orient all curves so that we get the absolute grading defined in Section \ref{abssec} by looking at $\HH :=\HH_{\zz}\cup \HH_D$, and write $\balpha = \{\alpha_1, \ldots, \alpha_{2k}, \alpha_1^c, \ldots, \alpha_{g-k}^c\}$. Let $\alpha_i^H$ be the result of performing a handleslide of $\alpha_i$ over $\alpha_j^c$, and let $\HH_D' = (\Sigma, \balpha^H, \bbeta, z)$ be the diagram resulting from that handleslide. The ordering and orientation on $\balpha$ induces an ordering and orientation on $\balpha^H$: each $\alpha_t^{c,H}$ is a small perturbation of the \emph{oriented} circle $\alpha_t^c$, for $t\neq i$ each $\alpha_t^{H}$ is an isotopic translate of the \emph{oriented} arc $\alpha_t$, and $\alpha_i^H$ is also obtained as in \cite[Section 6.3.2]{bfh2} and oriented so that each of the short Reeb chords in $\bdy \bar\Sigma$ running from $\alpha_i^H$ to $\alpha_i$ connects two points that inherit the same orientation (i.e. $\alpha_i^H$ is oriented according to our usual ``bottom to top" convention). We preserve the ordering and orientation on $\bbeta$.

We claim that the ordering and orientation on $\balpha^H$ and $\bbeta$ is compatible with the absolute grading obtained by closing off the Heegaard diagram to $\HH' := \HH_{\zz}\cup \HH_D'$. This is true since the handleslide in $\Sigma$ corresponds to a handleslide in the closed Heegaard surface  obtained by gluing in $\HH_{\zz}$.
If we think of the absolute $\Z/2$ grading on the closed diagram as corresponding to the canonical homology orientation for the closed $3$-manifold $H_{\zz}\cup Y(\HH)$, as  in Section \ref{closedgr},  then observe that the oriented handleslide
corresponds to a change of basis for $C_1(H_{\zz}\cup Y(\HH))\oplus C_2(H_{\zz}\cup Y(\HH))$ that replaces $[\widetilde\alpha_i]$ with $[\widetilde\alpha_i \pm \alpha_j^c]$, where $\widetilde\alpha_i$ is the circle in $\HH$ containing the arc $\alpha_i$. This change of basis preserves orientation, so we have the same homology orientation on the closed manifold as before the handleslide. Alternatively, if we think of the $\Z/2$ grading on $\hfhat$ for the closed diagram as the one coming from $\underline{\hf}^{\infty}(H_{\zz}\cup Y(\HH))$, then observe that the oriented handleslide has a corresponding triangle map on  $\underline{\hf}^{\infty}$ that preserves the grading induced by intersection signs,  since it counts index $0$ domains, and nearest generators have identical permutations and local intersection signs. This fact is spelled out properly in the next paragraph.

Recall the definition of the triangle map (\cite[Equation 6.33]{bfh2}):
\begin{align*}
F_{\balpha, \balpha^H, \bbeta}(\xx) &:= \sum_{\yy}\sum_{B\in\pi_2(\xx', \yy)}\sum_{\{\vec{\rho}|\ind(B,\vec{\rho}) = 0\}}\#(\mathcal M^B(\xx, \yy, \boldsymbol\theta; \vec{\rho})) a(-\vec{\rho})\yy \\
F_{\balpha, \balpha^H, \bbeta}(a\xx) &:= aF_{\balpha, \balpha^H, \bbeta}(\xx).
\end{align*}
In simple words, given $\xx\in\cfdhat(\Sigma, \balpha, \bbeta, z)$, we count holomorphic curves of index $0$ on $(\Sigma, \balpha, \bbeta, z)$ starting at the nearest generator $\xx'$ to $\xx$.
 For any generator $\xx$, the nearby generator $\xx'$ carries the same permutation and local intersection signs, and hence has the same grading. By Proposition \ref{indgr}, $m(\yy) = m(\xx') = m(\xx)$. In other words, the triangle map preserves the grading. 
 \end{proof}

\bibliographystyle{/Users/ina/Documents/work/hamsplain2}

\bibliography{/Users/ina/Documents/work/master}

\end{document}

%% file: macros.tex
\def\example{{\bf {\bigskip}{\noindent}Example: }}

\def\remark{{\bf {\bigskip}{\noindent}Remark. }}

\def\ackn{{\bf {\bigskip}{\noindent}Acknowledgments. }}
\def\bar{\overline}

\newcommand{\bdy}{\partial}

\newcommand{\az}{\mathcal{A}(\zz)}
\newcommand{\cala}{\mathcal{A}}
\newcommand{\ainf}{\mathcal{A}_\infty}

\newcommand{\sss}{{\bf s}}

\newcommand{\xx}{{\bf x}}
\newcommand{\yy}{{\bf y}}

\newcommand{\ta}{{\mathbb T_{\alpha}}}
\newcommand{\tb}{{\mathbb T_{\beta}}}

\DeclareMathOperator{\Int}{Int}

\DeclareMathOperator{\inv}{inv}
\DeclareMathOperator{\ind}{ind}
\DeclareMathOperator{\id}{id}

\newcommand{\HH}{\mathcal{H}}

\newcommand{\zz}{\mathcal Z}

\newcommand{\Z}{\mathbb Z}
\newcommand{\R}{\mathbb R}

\newcommand{\Q}{\mathbb Q}
\newcommand{\F}{\mathbb F}

\newcommand{\aaa}{\mathbf{a}}
\newcommand{\bbb}{\mathbf{b}}
\newcommand{\ccc}{\mathbf{c}}
\newcommand{\ddd}{\mathbf{d}}
\newcommand{\xxx}{\mathbf{x}}
\newcommand{\yyy}{\mathbf{y}}
\newcommand{\zzz}{\mathbf{z}}
\newcommand{\www}{\mathbf{w}}

\newcommand{\balpha}{\boldsymbol\alpha}
\newcommand{\bbeta}{\boldsymbol\beta}

\newcommand{\cf}{\mathit{CF}}
\newcommand{\hf}{\mathit{HF}}
\newcommand{\cfd}{\mathit{CFD}}
\newcommand{\cfa}{\mathit{CFA}}

\newcommand{\sfh}{\mathit{SFH}}
\newcommand{\cfhat}{\widehat{\cf}}
\newcommand{\hfhat}{\widehat{\hf}}
\newcommand{\cfdhat}{\widehat{\cfd}}
\newcommand{\cfahat}{\widehat{\cfa}}

\newcommand{\hha}{\HH^{\alpha}}
\newcommand{\hhb}{\HH^{\beta}}
\newcommand{\hz}{\HH_{\zz}}

\newcommand{\brho}{\boldsymbol\rho}

%% file: abs_background.tex
This section is a brief introduction to bordered Floer homology.

\subsection{The algebra $\az$}
 We describe the differential graded algebra $\az$ associated to the parametrized boundary of a $3$-manifold. For further details, see  \cite[Chapter 3]{bfh2}.

\begin{defn}\label{alg}
The \emph{strands algebra} $\cala(n,k)$ is a free $\Z/2$-module generated by partial permutations $a = (S, T, \phi)$, where $S$ and $T$ are $k$-element subsets of the set $[n]:= \{1, \ldots, n\}$ and $\phi:S\to T$ is a non-decreasing bijection. We let $\inv(a) = \inv(\phi)$ be the number of inversions of $\phi$, i.e. the number of pairs $i,j\in S$ with $i<j$ and $\phi(j)<\phi(i)$. Multiplication is given by 
\begin{displaymath}
(S, T, \phi)\cdot(U, V, \psi) = \left\{ \begin{array}{ll}
(S, V, \psi\circ\phi) & \textrm{if $T=U$, $\inv(\phi)+ \inv(\psi) = \inv(\psi\circ\phi)$}\\
0 & \textrm{otherwise.}
\end{array} \right.
\end{displaymath}
See  \cite[Section 3.1.1]{bfh2}.
We can represent a generator $(S, T, \phi)$ by a strands diagram of horizontal and upward-veering strands. See  \cite[Section 3.1.2]{bfh2}.
The differential of $(S, T, \phi)$ is the sum of all possible ways to ``resolve" an inversion of $\phi$ so that $\inv$ goes down by exactly $1$. Resolving an inversion $(i,j)$ means switching $\phi(i)$ and $\phi(j)$, which graphically can be seen as smoothing a crossing in the strands diagram.
\end{defn}

The ring of idempotents $\mathcal I(n,k)\subset \cala(n,k)$ is generated by all elements of the form $I(S) := (S, S, \textrm{id}_S)$ where $S$ is a $k$-element subset of $[n]$. 
 
\begin{defn} A \emph{pointed matched circle} is a quadruple $\zz = (Z, {\bf a}, M, z)$ consisting of an oriented circle $Z$, a collection of $4k$ points ${\bf a} = \{a_1, \ldots, a_{4k}\}$ in $Z$, a \emph{matching} of ${\bf a}$, i.e., a $2$-to-$1$ function $M:{\bf a}\to [2k]$, and a basepoint $z\in Z\setminus {\bf a}$. We require that performing oriented surgery along the $2k$ $0$-spheres $M^{-1}(i)$ yields a single circle. 
\end{defn}

A matched circle specifies a handle decomposition of an oriented surface $F(\zz)$ of genus $k$: take a $2$-dimensional $0$-handle with boundary $Z$, $2k$  $1$-handles attached along the pairs of matched points, and a $2$-handle attached to the resulting boundary. 

If we forget the matching on the circle for a moment, we can view $\cala(4k) = \bigoplus_{i}\cala(4k, i)$ as the algebra generated by certain sets of Reeb chords in $(Z\setminus z, {\bf a})$: We can view a set ${\boldsymbol{\rho}}$ of Reeb chords, no two of which share initial or final endpoints, as a strands diagram of upward-veering strands. For such a set ${\boldsymbol{\rho}}$, we define the \emph{strands algebra element associated to ${\boldsymbol{\rho}}$} to be the sum of all ways of consistently adding horizontal strands to the diagram for ${\boldsymbol{\rho}}$, and we denote this element by $a_0({\boldsymbol{\rho}})\in \cala(4k)$. The basis over $\Z/2$ from Definition \ref{alg} is in this terminology the non-zero elements of the form $I(S)a_0({\boldsymbol{\rho}})$, where $S\subset \bf a$.

For a subset ${\bf s}$ of $[2k]$, a \emph{section} of ${\bf s}$ is a set  $S\subset M^{-1}({\bf s})$, such that $M$ maps $S$ bijectively to ${\bf s}$. To each ${\bf s}\subset [2k]$ we associate an idempotent in $\cala (4k)$ given by
$$I({\bf s}) = \sum_{S \textrm{ is a section of } {\bf s}} I(S).$$
Let $\mathcal I(\zz)$ be the subalgebra generated by all $I({\bf s})$, and let ${\bf I} = \sum_{\bf s}I({\bf s})$.
\begin{defn}
The \emph{algebra $\az$ associated to a pointed matched circle $\zz$} is the subalgebra of $\cala(4k)$ generated (as an algebra) by $\mathcal I(\zz)$ and by all $a({\boldsymbol{\rho}}) :={\bf I}a_0({\boldsymbol{\rho}})\bf{ I}$. We refer to $a({\boldsymbol{\rho}})$ as the \emph{algebra element associated to ${\boldsymbol{\rho}}$}.
\end{defn}

\subsection{Type $D$ structures, $\cala_\infty$-modules, and tensor products}

We recall the definitions of the algebraic structures used in \cite{bfh2}.  For a beautiful, terse description of type $D$ structures and their basic properties, see \cite[Section 7.2]{bs}, and for a more general and detailed description of $\cala_\infty$ structures,  see \cite[Chapter 2]{bfh2}. 

Let $A$ be a unital differential graded algebra with differential $d$ and multiplication $\mu$ over a base ring ${\bf k}$. In this paper, ${\bf k}$ will always be a direct sum of copies of $\F_2 = \Z/2\Z$. When the algebra is $\cala(\mathcal Z)$, the base ring for all modules and tensor products is $\mathcal I(\mathcal Z)$.

A \emph{(right) $\cala_\infty$-module} over $A$ is a graded module $M$ over ${\bf k}$, equipped with maps 
$$m_i: M\otimes A^{\otimes (i-1)}\to M[2-i],$$ 
satisfying the compatibility conditions
\begin{align*}
0&= \sum_{i+j = n+1}m_i(m_j(\xxx, a_1, \ldots, a_{j-1}), \ldots , a_{n-1})\\
&+ \sum_{i=1}^{n-1} m_n(\xxx, a_1,\ldots, a_{i-1}, d(a_i),\ldots, a_{n-1})\\
&+ \sum_{i=1}^{n-2} m_{n-1}(\xxx, a_1,\ldots, a_{i-1}, (\mu(a_i,a_{i+1})),\ldots, a_{n-1})
\end{align*}
and the unitality conditions $m_2(\xxx, 1) = \xxx$ and $m_i(\xxx, a_1,\ldots, a_{i-1})=0$ if $i>2$ and some $a_j =1$. We say that $M$ is \emph{bounded} if  $m_i=0$ for all sufficiently large $i$.

A \emph{(left) type $D$ structure}  over $A$ is a graded module $N$ over the base ring, equipped with a homogeneous map
$$\delta:N\to (A\otimes N)[1]$$
satisfying the compatibility condition
$$(d\otimes \id_N)\circ \delta + (\mu\otimes \id_N)\circ (\id_A\otimes \delta)\circ \delta= 0.$$
We can define maps
$$\delta_k:N\to (A^{\otimes k}\otimes N)[k]$$
inductively by
$$\delta_k = \left\{  \begin{array}{ll}
\id_N & \textrm{ for } k=0\\
(\id_A\otimes \delta_{k-1})\circ \delta & \textrm{ for } k\geq 1 
\end{array}\right.$$

 A type $D$ structure is said to be \emph{bounded} if for any $\xxx \in N$, $\delta_i(\xxx) = 0$ for all sufficiently large $i$.

If $M$ is a right $\cala_\infty$-module over $A$ and $N$ is a left type $D$ structure, and at least one of them is bounded, we can define the \emph{box tensor product} $M\boxtimes N$ to be the  vector space $M\otimes N$ with differential 
$$\bdy: M\otimes N \to M\otimes N[1]$$
defined by 
$$\bdy = \sum_{k=1}^{\infty}(m_k\otimes \id_N)\circ(\id_M\otimes \delta_{k-1}).$$
The boundedness condition guarantees that the above sum is finite. In that case $\bdy^2= 0$ and $M\boxtimes N$ is a graded chain complex. In general (boundedness is not required), one can think of a type $D$ structure as a left $\ainf$ module, and take an $\ainf$ tensor product $\widetilde\otimes$, see \cite[Section 2.2]{bfh2}.

Given two differential graded algebras, four types of bimodules can be defined in a similar way. We omit those definitions and refer the reader to \cite[Section 2.2.4]{bimod}.

\subsection{Bordered three-manifolds, Heegaard diagrams, and their modules}\label{cfacfd_background}

A \emph{bordered $3$-manifold} is  a triple $(Y, \zz, \phi)$, where $Y$ is a compact, oriented $3$-manifold with connected boundary $\bdy Y$, $\zz$ is a pointed matched circle, and $\phi: F(\zz)\to \bdy Y$ is an orientation-preserving homeomorphism. A bordered $3$-manifold may be represented by a \emph{bordered Heegaard diagram}  $\mathcal H =(\Sigma, \balpha, \bbeta, z)$, where $\Sigma$ is an oriented surface of some genus $g$  with one boundary component, $\bbeta$ is a set of pairwise-disjoint, homologically independent circles in $\Int(\Sigma)$, $\balpha$ is a $(g+k)$-tuple of pairwise-disjoint curves in $\Sigma$, split into $g-k$ circles in $\Int(\Sigma)$, and $2k$ arcs with boundary on $\bdy \Sigma$, so that they are all homologically independent in $H_1(\Sigma, \bdy \Sigma)$, and $z$ is a point on $(\bdy\Sigma)\setminus (\balpha\cap \bdy \Sigma)$.
The boundary $\bdy\mathcal H$ of the Heegaard diagram has the structure of a pointed matched circle, where two points are matched if they belong to the same $\alpha$-arc. 
We can see how a bordered Heegaard diagram $\mathcal H$ specifies a bordered manifold  in the following way. Thicken up the surface  to $\Sigma\times [0,1]$, and attach a three-dimensional two-handle to each circle $\alpha_i\times \{0\}$, and a three-dimensional two-handle to each $\beta_i\times \{1\}$. Call the result $Y$, and let $\phi$ be the natural identification of $F(\bdy \mathcal H)$ with $\bdy Y$ induced by the $\alpha$-arcs. Then $(Y, \bdy \mathcal H, \phi)$ is the bordered 3-manifold for $\mathcal H$.

A \emph{generator} of a bordered Heegaard diagram $\mathcal H = (\Sigma, \balpha, \bbeta, z)$ of genus $g$ is a $g$-element subset $\xxx = \{x_1, \ldots, x_g\}$ of $\balpha\cap \bbeta$, such that there is exactly one point of $\xxx$  on each $\beta$-circle, exactly one point on each $\alpha$-circle, and at most one point on each $\alpha$-arc. Let $\mathfrak S(\mathcal H)$ denote the set of generators. Given $\xxx\in \mathfrak S(\mathcal H)$, let $o(\xxx)\subset [2k]$ denote the set of $\alpha$-arcs occupied by $\xxx$, and let $\bar o(\xxx)=[2k]\setminus o(\xxx)$ denote the set of unoccupied arcs.

Fix generators $\xxx$ and $\yyy$, and let $I$ be the interval $[0,1]$. Let $\pi_2(\xxx, \yyy)$, the \emph{homology classes from $\xxx$ to $\yyy$}, denote the elements of 
$$H_2(\Sigma\times I\times I, ((\balpha\times \{1\}\cup \bbeta\times \{0\}\cup(\bdy\Sigma\setminus z)\times I)\times I)\cup (\xxx\times I\times \{0\})\cup (\yyy\times I\times\{1\}))$$
which map to the relative fundamental class of $\xxx\times I\cup \yyy\times I$ under the composition of the boundary homomorphism and collapsing the rest of the boundary.

A homology class $B\in \pi_2(\xxx, \yyy)$ is determined by its \emph{domain}, the projection of $B$ to $H_2(\Sigma, \balpha\cup\bbeta\cup \bdy\Sigma)$. We can interpret the domain of $B$ as a linear combination of the components, or \emph{regions}, of $\Sigma\setminus (\balpha\cup \bbeta)$.

Concatenation at $\yyy\times I$, which corresponds to addition of domains, gives a product $\ast:\pi_2(\xxx, \yyy)\times \pi_2(\yyy, \www)\to \pi_2(\xxx, \www)$. This operation turns $\pi_2(\xxx, \xxx)$ into a group called the group of \emph{periodic domains}, which is naturally isomorphic to $H_2(Y, \bdy Y)$.

To a bordered Heegaard diagram $\mathcal H$, we associate either a left type $D$ structure $\cfdhat (\mathcal H)$ over $\mathcal A(-\bdy\mathcal H)$, or a right $\mathcal A_{\infty}$-module $\cfahat (\mathcal H)$ over $\mathcal A(\bdy\mathcal H)$, as follows.

Let $X(\mathcal H)$ be the $\F_2$ vector space spanned by $\mathfrak S(\mathcal H)$. Define $I_D(\xx)= \bar o(\xx)$. We define an action  on $X(\mathcal H)$ of $\mathcal I(-\bdy\mathcal H)$ by
$$I(\sss)\cdot \xx = \left\{  \begin{array}{ll}
\xx & \textrm{ if } I(\sss) = I_D(\xx)\\
0 & \textrm{ otherwise. } 
\end{array}\right.$$
Then $\cfdhat(\mathcal H)$ is the left $\mathit{dg}$ module defined as an $\mathcal A(-\bdy\mathcal H)$-module by 
$$\cfdhat (\mathcal H) = \mathcal A(-\bdy\mathcal H)\otimes_{\mathcal I(-\bdy\mathcal H)}X(\mathcal H),$$
with  differential is given by 
$$\partial(I_D(\xx)\otimes \xx) = \sum_{\yy\in \mathfrak S(\mathcal H)}a_{\xx,\yy}\otimes \yy,$$
where $a_{\xx,\yy}$ counts certain holomorphic representatives of the homology classes $B\in \pi_2(\xxx, \yyy)$ with asymptotics $\brho$, and then extended by linearity and the Leibnitz rule to all of $\cfdhat(\mathcal H)$. 
We will not describe the count $a_{\xx,\yy}$ fully, but only remark that for a pair $(B, \brho)$ to  contribute to $a_{\xx, \yy}$, a certain  moduli space needs to have expected dimension zero. This can only happen if $\ind(B, \brho) = 1$. We discuss this index in Section \ref{inv}.

As a type $D$ structure, $\cfdhat(\mathcal H)$ is the $\mathcal I(-\bdy\mathcal H)$-module $X(\mathcal H)$ with 
structure map $\delta:\cfdhat(\mathcal H)\to\cala(-\bdy\HH)\otimes_{\mathcal I(-\bdy\mathcal H)} \cfdhat(\mathcal H)$ defined by
\[\delta(\xx) = \partial(I_D(\xx)\otimes \xx).\]

Define $I_A(\xx)=  o(\xx)$. The module $\cfahat(\mathcal H)$ is generated over $\F_2$ by $X(\mathcal H)$, and the right action of $\mathcal I(\bdy \mathcal H)$ on $\cfahat(\mathcal H)$ is defined by
$$\xx\cdot I(\sss) = \left\{  \begin{array}{ll}
\xx & \textrm{ if } I(\sss) = I_A(\xx)\\
0 & \textrm{ otherwise. } 
\end{array}\right.$$
The $\ainf$ multiplication maps count certain holomorphic representatives of the homology classes defined in this section \cite[Definition 7.3]{bfh2}. Since we only discuss $\cfdhat$ in this paper, we omit a more precise definition of $\cfahat$.

\subsection{Gradings}\label{m}

It is easy to demonstrate that the algebra $\az$ has no differential $\Z$-grading with respect to which the generators are homogeneous. In \cite{bfh2}, the authors construct a grading on  $\az$ by a non-commutative group denoted by $G'(n)$, which is a central extension by $\Z$ of the relative homology group $H_1(Z\setminus z, \aaa)$. If certain choices, that we refer to as ``refinement data", are made, this grading descends to a grading in a smaller group $G(\zz)$, which is a central extension of $H_1(F)$. This is the Heisenberg group associated to the intersection form of $F$. With an additional choice of a base generator in each $\mathrm{spin}^c$ structure, one obtains a grading by a $G'(n)$-set, respectively a $G(\zz)$-set,  on any left or right module over $\az$. 

Up to $\cala_{\infty}$ homotopy equivalence, the graded $\cala_{\infty}$ module $\cfahat(\HH, \mathfrak s)$, and similarly the graded type $D$ structure  $\cfdhat(\HH, \mathfrak s)$, is independent of most of the choices made in its definition. However, it still depends on the refinement data. In addition, the set gradings in \cite{bfh2} are only defined within a specific $\mathrm{spin}^c$ structure. Given two bordered manifolds that can be glued along their boundary, the pairing theorem \cite[Theorem 10.42]{bfh2} provides a relation between the set-graded modules corresponding to the two bordered manifolds and the Maslov-graded Heegaard Floer complex for their union.